\documentclass[12pt]{article}
\usepackage[english]{babel}
\usepackage{amsfonts,amsmath,amsxtra,amsthm,amssymb,latexsym}
\usepackage{color,soul}

\newcommand{\N}{{\mathbb N}}
\newcommand{\R}{{\mathbb R}}
\newcommand{\C}{{\mathbb C}}
\newcommand{\Z}{{\mathbb Z}}
\newcommand{\T}{{\mathbb T}}

\newcommand{\kA}{{\mathcal A}}
\newcommand{\kB}{{\mathcal B}}
\newcommand{\kS}{{\mathcal S}}
\newcommand{\kM}{{\mathcal M}}
\newcommand{\kH}{{\mathcal H}}
\newcommand{\kK}{{\mathcal K}}

\newcommand{\kX}{{\mathcal X}}
\newcommand{\kF}{{\mathcal F}}
\newcommand{\kL}{{\mathcal L}}
\newcommand{\kE}{{\mathcal E}}

\newcommand{\La}{\Lambda}
\newcommand{\la}{\lambda}

\newcommand\dom{\operatorname{dom}}

\newcommand{\dist}{\mathop{\rm dist}\nolimits}

\newcommand{\const}{\mathop{\rm const}\nolimits}

\newtheorem{theorem}{Theorem}[section]

 \newtheorem{corollary}[theorem]{Corollary}
 \newtheorem{lemma}[theorem]{Lemma}
 \newtheorem{proposition}[theorem]{Proposition}
 \theoremstyle{definition}
 \newtheorem{definition}[theorem]{Definition}
 \theoremstyle{remark}
 \newtheorem{remark}[theorem]{Remark}
 \newtheorem{example}[theorem]{Example}
 \numberwithin{equation}{section}

\author{L. Golinskii, M. Malamud and L. Oridoroga }

\begin{title}
{Schoenberg matrices of radial positive definite functions and Riesz sequences in $L^2(\R^n)$.}
\end{title}
\date{}
\begin{document}
\maketitle

\begin{abstract}
Given a function $f$ on the positive half-line $\R_+$ and a
sequence (finite or infinite) of points $X=\{x_k\}_{k=1}^\omega$ in
$\R^n$, we define and study matrices
$\kS_X(f)=\|f(|x_i-x_j|)\|_{i,j=1}^\omega$ called Schoenberg's
matrices. We are primarily interested in those matrices which
generate bounded and invertible linear operators $S_X(f)$ on $\ell^2(\N)$.
We provide conditions on $X$ and
$f$ for the latter to hold. If $f$ is an $\ell^2$-positive definite
function, such conditions are given in terms of the Schoenberg
measure $\sigma(f)$. We also approach Schoenberg's matrices from the
viewpoint of harmonic analysis on $\R^n$, wherein the notion of the
strong $X$-positive definiteness plays a key role. In
particular, we prove that \emph{each radial $\ell^2$-positive definite
function is strongly  $X$-positive definite} whenever $X$ is
separated. We also implement a ``grammization'' procedure
for certain positive definite Schoenberg's matrices. This leads to Riesz--Fischer and
Riesz sequences (Riesz bases in their linear span) of the form $\kF_X(f)=\{f(x-x_j)\}_{x_j\in X}$
for certain radial functions $f\in L^2(\R^n)$. Examples of Schoenberg's operators with various spectral
properties are presented.
\end{abstract}


\textbf{Key  words}. Infinite matrices, Schur test,
Toeplitz operators, Riesz bases, completely monotone functions, Gramm matrices

\renewcommand{\contentsname}{Contents}
\tableofcontents

%%%%%%%%%%%%%%%%%%%%%%%%%%%%%%%
%%
\section{Introduction}

Positive definite functions have a long history, entering as an important chapter in all treatments
of harmonic analysis. They can be traced back to papers of Carath\'eodory, Herglotz, Bernstein, culminating
in Bochner's celebrated theorem from 1932--1933. See definitions in Section \ref{prelimpdf}.

In this paper we will be dealing primarily with radial positive definite functions (RPDF).
RPDF's have significant applications in probability theory, statistics, and approximation theory,
where they occur as the characteristic functions or Fourier transforms of spherically symmetric probability
distributions, the covariance functions of stationary and isotropic random fields, and the radial basis functions
in scattered data interpolation. We denote the class of RPDF's by $\Phi_n$.

We stick to the standard notation for the inner product $(u,v)_n=(u,v)=u_1v_1+\ldots+u_nv_n$ of two vectors
$u=(u_1,\ldots,u_n)$ and $v=(v_1,\ldots,v_n)$ in $\R^n$, and $ |u|_n=|u|=\sqrt{(u,u)} $ for the Euclidean norm of~$u$.
We want to emphasize from the outset that throughout the whole paper $n$ is an arbitrary and fixed positive integer.

%%%%%%%%%%%%%%%%%%%%%%%%%%%%%%%%%%%
   \begin{definition} Let $n\in \N$.
A real-valued and continuous function $f$ on $\R_+=[0,\infty)$ is called a {\it radial positive definite function}, if
for an arbitrary finite set $\{x_1,\dots,x_m\}$, $x_k\in \R^n$, and $\{\xi_1,\dots,\xi_m\}\in\C^m$
\begin{equation}\label{positiv}
\sum_{k,j=1}^{m} f(|x_k-x_j|)\xi_j\overline{\xi}_k\ge 0.
\end{equation}
\end{definition}

The characterization of radial positive definite functions is a classical
result due to I. Schoenberg \cite{Sch38, Sch38_1} (see, e.g., \cite[Theorem~5.4.2]{Akh65}).

%%%%%%%%%%%%%%%%%%%%%%%%%%%%%%%%%
  \begin{theorem}\label{schoenbergtheorem}
 A function  $f\in\Phi_n$, $f(0)=1$, if and only if there exists
 a probability measure $\nu$ on $\R_+$ such that
  \begin{equation}\label{schoenberg1}
f(r) = \int_{0}^{\infty}\Omega_n(rt)\,\nu(dt), \qquad r\in \R_+,
   \end{equation}
where
  \begin{equation}\label{kernel}
  \Omega_n(s):=\Gamma(q+1)\,
  \left(\frac{2}{s}\right)^q\,J_q(s)=
  \sum_{j=0}^{\infty}\frac{\Gamma(q+1)}{j!\,\Gamma(j+q+1)}\,
  \left(-\frac{s^2}{4}\right)^j, \quad q:=\frac{n}2-1,
 \end{equation}
$J_q$ is the Bessel function of the first kind and order $q$. Moreover,
   \begin{equation}\label{fouriersphere}
\Omega_n(|x|) = \int_{S^{n-1}}e^{i(u,x)}\sigma_n(du), \qquad
x\in\R^n,
  \end{equation}
where $\sigma_n$ is the normalized surface measure on the unit sphere
$S^{n-1}\subset\R^n$.
  \end{theorem}

The first three functions $\Omega_n$, $n=1,2,3,$ can be computed as
 \begin{equation}\label{3.13}
   \Omega_1(s)=\cos s,\quad  \Omega_2(s)=J_0(s),\quad  \Omega_3(s)=\frac{{\rm sin} s}{s}\,.
 \end{equation}

\medskip

The main object under consideration in this paper arises from the definition of RPDF's.

\begin{definition}\label{defboch}
Let $X=\{x_k\}_{k=1}^\omega\subset \R^n$ be a (finite or infinite) set
of distinct points in $\R^n$ and let $f$ be a real-valued function defined on
the right half-line $\R_+$. A matrix (finite or infinite)
\begin{equation}\label{bochmatr}
\kS_X(f):=\|f(|x_i-x_j|)\|_{i,j=1}^\omega, \qquad \omega\le\infty,
\end{equation}
will be called a {\it Schoenberg matrix} generated by the set $X$
and the function $f$. This function is referred to as the {\it
Schoenberg symbol}.
\end{definition}
It is clear that $\kS_X(f)$ is a Hermitian (real symmetric) matrix. By the definition, a function $f\in\Phi_n$
if for each finite set $X\subset\R^n$ the Schoenberg matrix $\kS_X(f)$ is nonnegative, $\kS_X(f)\ge0$.

\smallskip

We undertake a detailed study of Schoenberg's matrices from two different points of view.
The first one, considered in Section \ref{shoenoper}, comes from operator theory.

If the columns of $\kS_X(f)$  are in $\ell^2 := \ell^2(\N)$, then one
can  associate  a minimal symmetric operator $S_X(f)$ with
$\kS_X(f)$ in a natural way. We call it a {\it Schoenberg operator}. If $S_X(f)$ appears to be bounded,
a matrix $\kS_X(f)$ (admitting some abuse of language) will be called bounded.
The  \emph{first  main goal of the paper is to  find necessary and
sufficient conditions on $X$ and $f$, which ensure that the matrix
$\kS_X(f)$ is  bounded}. We also suggest conditions on $X$ and $f$ for $S_X(f)$ to be
invertible, i.e., to have a bounded inverse.

Throughout the paper we always assume that $X$ is a {\it separated set}, i.e.,
\begin{equation}\label{separat}
d_*=d_*(X):=\inf_{i\not=j}|x_i-x_j|>0,
\end{equation}
(the term {\it uniformly discrete} is also in common usage). We denote by $\kX=\kX_n$ the class of all separated sets
$X\subset\R^n$ and by $\kL=\kL(X)$ a linear span of $X$, a subspace in $\R^n$ of dimension $d=d(X)=\dim\kL\le n$.
With no loss of generality we can assume that $x_1=0$.

Next,  denote by $\kM_+$ the following class of functions:
\begin{equation}\label{function}
f\in\kM_+: \ \ f\ge 0, \qquad f\downarrow, \qquad f(0)=1.
\end{equation}

With this preparation our main result on boundedness  of $\kS_X(f)$
reads as follows.
%%%%%%%%%%%%%%%%%%%%%%%%%%%%%%%%%%%%%%%%%%%
   \begin{theorem}\label{intrbochbound}
Let $f\in\kM_+$, $X\in\kX_n$ and let $d = \dim\kL(X)$.
\begin{itemize}
  \item [\em (i)]
 If $t^{d-1}f(\cdot)\in L^1(\R_+)$, then the Schoenberg matrix  $\kS_X(f)$ is bounded
on $\ell^2$ and
\begin{equation}
\|S_X(f)\|\le 1+d^2\biggl(\frac{5}{d_*(X)}\biggr)^d\,\int_0^\infty
t^{d-1}\,f(t)\,dt.
\end{equation}
\item [\em (ii)]
Moreover, $S_X(f)$ has a bounded inverse  whenever, in addition,
\begin{equation}\label{boundforinverse}
d_*(X)>5d^{2/d}\,\|t^{d-1}f\|_{L^1(\R_+)}^{1/d}.
\end{equation}
  \item [\em (iii)]
Conversely, let $S_Y(f)$ be bounded for at least one
$\delta$-regular set $Y$. Then $t^{d-1}f(\cdot)\in L^1(\R_+)$.
\end{itemize}
\end{theorem}

Concerning regular sets see Definition \ref{regularset}. For instance, $X= \delta\Z^n$  is $\delta$-regular.

In particular, Theorem \ref{intrbochbound}  completely describes bounded operators $S_X(f)$  with symbols  $f$ from the classes
$\Phi_{\infty}(\alpha)$ defined below in Section \ref{Subcl_Alpha_prelimpdf}.

We also discuss the Fredholm property of the Schoenberg operators, precisely, the case when $S_X(f)=I+T$,
$T$ is a compact operator on $\ell^2$.

An interesting example of general Schoenberg operators arises when the set $X$ is a Toeplitz sequence, that is, $|x_i-x_j|=|i-j|$ for
all $i,j\in\N$. Such operators will be called the {\it Schoenberg--Toeplitz operators}. We obtain necessary and sufficient conditions for the
Schoenberg--Toeplitz operators with special symbols to be bounded and describe their spectra  in terms of Schoenberg's symbols. We show
that such (possibly unbounded) operators are always self-adjoint.

\medskip

Our second viewpoint on Schoenberg's matrices is related to harmonic analysis on $\R^n$.

The main result of Section \ref{classes} is related to the notion of the strong $X$-positive definiteness.

   \begin{definition}\label{def-X-posit_definit}
Let $f\in\Phi_n$ and $X=\{x_k\}_{k\in\N}\subset\R^n$. We say that $f$ is {\it strongly $X$-positive definite}
(or the Schoenberg matrix $\kS_X(f)$ is positive definite)
if for each set $\xi=\{\xi_1,\dots,\xi_m\}\in \C^m\setminus\{0\}$
and any finite set $\{x_j\}_{j=1}^m$  of distinct points $x_j\in X$
there exists a constant $c=c(X)>0$, independent of $\xi$ and $m$ such that
   \begin{equation}\label{stronglyxpositive}
\sum_{k,j=1}^{m}f(|x_k-x_j|){\xi}_j\overline{\xi_k} \ge c \sum_{k=1}^m |\xi_k|^2.
   \end{equation}
The same definition with obvious changes applies to general (not necessarily radial) positive definite functions.

We say that $f$ is {\it strictly $X$-positive definite} if for each
$m\in \N$ and $\xi=\{\xi_1,\dots,\xi_m\}\in \C^m\setminus\{0\}$
inequality  \eqref{stronglyxpositive} holds with $c=0.$
    \end{definition}

Equivalently, $f$ is strictly $X$-positive definite, if for any
finite subset $Y\subset X$ the Schoenberg matrix $\kS_Y(f)$ is
non-singular, i.e., the minimal eigenvalue $\lambda_{min}(\kS_Y(f))$
of $\kS_Y(f)$ is positive, and strongly  $X$-positive definite, if
$\kS_Y(f)$ are ``uniformly positive definite'', that is,
$$
\inf_{Y\subset X} \lambda_{min}(\kS_Y(f))>0,
$$
where the infimum is taken over all finite subsets $Y\subset X$.

The notion of strong $X$-positive definiteness makes sense for any $f\in\Phi_n$ regardless of whether
the Schoenberg operator $S_X(f)$ is defined or not. In the former case
\emph{the strong $X$-positive definiteness of $f$ is identical to positive
definiteness of $S_X(f)$}, i.e., validity of the inequality
  \begin{equation}\label{1.10Intro}
\bigl(S_{X}(f)h, h)\ge\varepsilon|h|^2, \qquad h \in\dom
S_X(f)\subset \ell^2,\qquad \varepsilon>0.
    \end{equation}
with some $\varepsilon>0$ independent of $h.$
So Definition \ref{def-X-posit_definit}  merely extends a property \eqref{1.10Intro} of
$S_X(f)$, when the latter exists, to the case of an arbitrary Schoenberg
matrix $\kS_X(f)$, not necessarily generating an operator in $\ell^2$.

Each strongly $X$-positive definite function $f$  is also strictly $X$-positive definite.
For finite sets $X$  both notions are equivalent due to the compactness of the balls in $\C^m$.
The following problem seems to be important and difficult.

{\bf{Problem I.}} {\it Let $f$ be a radial positive definite function on
$\R^n$. Characterize those countable subsets $X$ of $\R^n$ for which
$f$ is strongly $X$-positive definite. }

It was proved in \cite{sun93} (see also  \cite[Theorem 3.6]{GMZ11})
that each function $f\in\Phi_n$, $n\ge2$, is strictly $X$-positive
definite for any set $X$ of distinct points in $\R^n$. This fact has
been heavily exploited in \cite{GMZ11} for  investigation of certain
spectral properties of $2D$ and $3D$ Schr\"odinger operator with a
{\it finite number} of point interactions.  On the other hand, if a
radial positive definite function is $X$-strongly positive definite,
then $X$ is necessarily separated (see Proposition \ref{ftsepar}).

Our \emph{second main goal  is to  give  a partial solution to Problem I}.
Heading to the solution of this problem we prove the following result.
\begin{theorem}\label{Intro_stronglydef}
Let $(\const \not = )f\in\Phi_n$, $n\ge2$, with the representing measure $\nu = \nu(f)$
from $\eqref{schoenberg1}$. If $\nu$ is equivalent to the Lebesgue measure on $\R_+$, then $f$ is strongly
$X$-positive definite for each $X\in\kX_n$.
   \end{theorem}

Actually, the most complete result on the strong $X$-positivity and the boundedness
of $S_X(f)$ is obtained for the class $\Phi_\infty :=\bigcap\limits_{n\in\mathbb{N}}\Phi_n$
and its subclasses $\Phi_\infty(\alpha)$, $\alpha \in (0,2]$ defined in the next section.
It looks as follows.

%%%%%%%%%%%%%%%%%%%%%%%%%%%%%%%%%
   \begin{theorem}\label{intrinftystrong}
Let $f\in \Phi_\infty(\alpha)$, $0<\alpha\le2$ and $X\in\kX_n$. Then

\begin{itemize}
  \item [\em (i)]
  $f$  is strongly $X$-positive definite. In particular, if $\kS_X(f)$  generates an
  operator $S_X(f)$ on $\ell^2$, then it is positive definite and so invertible.

  \item [\em (ii)]
If the Schoenberg measure $\sigma=\sigma_f$ in $\eqref{alphacm}$ satisfies
\begin{equation}\label{intralphabound1}
\int_0^\infty s^{-\frac{d}{\alpha}}\,\sigma(ds)<\infty, \qquad
d=\dim\kL(X),
\end{equation}
then the Schoenberg matrix $\kS_X(f)$  generates a bounded
$($necessarily invertible$)$ operator.

  \item [\em (iii)]
Conversely, let $S_Y(f)$ be bounded for at least one
$\delta$-regular set $Y$. Then \eqref{intralphabound1} holds.
\end{itemize}
  \end{theorem}

\medskip

The concept of ``grammization'' plays a key role in the rest of the
Section \ref{classes}.

It is a common knowledge that every positive matrix is the Gramm matrix of a certain system of vectors
\begin{equation}
\kA=\|a_{ij}\|_{i,j\in\N}\ge0 \Leftrightarrow \kA=\|(\varphi_i, \varphi_j)\|_{i,j\in\N}=:Gr(\{\varphi_k\}_{k\in\N}, \kH)
\end{equation}
$\{\varphi_k\}_{k\in\N}$ are vectors in a Hilbert space $\kH$.
According to the classical result of Bari, the property of a Gramm matrix $Gr\{\varphi_k\}_{k\in\N}$ to generate a
bounded and invertible operator on $\ell^2$ amounts  to the sequence $\{\varphi_k\}_{k\in\N}$ to be a Riesz sequence
(Riesz basis in its linear span).

The main applications of  Theorems \ref{intrbochbound} and \ref{Intro_stronglydef} are based on the grammization
procedure and concern Riesz--Fischer and Riesz sequences of shifts $\kF_X(f)=\{f(\cdot-x_j)\}_{j\in\N}$,
$X=\{x_j\}_{j\in\N}\subset\R^n$, of certain radial functions $f\in L^2(\R^n)$.

   \begin{proposition}\label{Intro prop4.21}
Let $f\in L^2(\R^n)$ be a real-valued and radial function such that
its Fourier transform $\widehat f\not=0$ a.e., and
$X=\{x_j\}_{j\in\N}\subset \R^n$. Then the following statements are
equivalent.
\begin{itemize}
  \item [\em (i)]   $\kF_X(f)$  forms a Riesz--Fischer  sequence in $L^2(\R^n);$
  \item [\em (ii)] $\kF_X(f)$ is uniformly minimal in $L^2(\R^n);$
  \item [\em (iii)]  $X$ is a separated set, i.e., $d_*(X)>0$.
\end{itemize}
     \end{proposition}
%%%%%%%%%%%%%%%%%%%%%%%%%%%%%%%%%%%%%%%%%%%%%%%%%%%%%
\begin{theorem}\label{theorem3.25}
Let $f\in L^2(\R^n)$ be a real-valued and radial function such that
its Fourier transform $\widehat f\not=0$ a.e. and
$X=\{x_j\}_{j\in\N}\subset \R^n$. Let $F$ and $F_0$ be defined as
  \begin{equation}\label{introradial1}
F(t)=(2\pi)^{n/2}|\widehat f(t)|^2 = F_0(|t|), \qquad \widehat F(t)=\widetilde F_0(|t|),
  \end{equation}
and assume that for some majorant $h\in\kM_+$ $\eqref{function}$ the
relations
\begin{equation}\label{majorant}
|\widetilde F_0(s)|\le h(s), \qquad s^{n-1} h(s)\in L^1(\R_+)
\end{equation}
hold. Then the following statements are equivalent.
\begin{itemize}
  \item [\em (i)]   $\kF_X(f)$  forms a Riesz sequence in $L^2(\R^n);$
  \item [\em (ii)]  $\kF_X(f)$ forms a basis in its linear span$;$
  \item [\em (iii)] $\kF_X(f)$ is uniformly minimal in $L^2(\R^n);$
  \item [\em (iv)]  $X$ is a separated set, i.e., $d_*(X)>0$.
\end{itemize}
\end{theorem}

The idea of the proof is related to the fact that the system
$\kF_X(f)$ performs the grammization of a certain Schoenberg's
matrix. So once we show that the latter generates a bounded and
invertible operator on $\ell^2$, the result is immediate  from the
Bari theorem. Thereby we make up a bridge between Riesz sequences
and Gramm matrices on the one hand and Schoenberg's matrices and
operators on the other hand.

We consider a number of examples which satisfy the assumptions of Proposition \ref{Intro prop4.21} and
Theorem~\ref{theorem3.25}. Among them
\begin{equation}\label{mainsystemIntro1}
f(x)=f_a(x)=e^{-a|x|^2}, \qquad  f(x)=f_{a,\mu}(x)=\biggl(\frac{a}{|\,x|}\,\biggr)^{\mu}\,K_\mu(a|\,x|),
\end{equation}
where $K_\mu$ is the modified Bessel function of the second kind and order $\mu$, $0\le\mu<n/4$.

Let us emphasize, that our choice of the second
system in \eqref{mainsystemIntro1} is also motivated  by
applications to elliptic operators with point interactions, since
the functions $f_{a,\mu}(\cdot-x_j)$ occur naturally in the spectral
theory of such operators for certain other values of $\mu$. We hope
to continue the study of this subject in our forthcoming papers.

It is worth stressing that in the abstract setting the uniform minimality is much
weaker than the Riesz sequence property. Nonetheless the equivalence of
these properties is well-known for certain classical systems:

(i) Exponential system $\{e^{i\lambda_k x}\}_{\lambda_k\in\Lambda}$
in $L^2[0,a)$,  $a\le\infty$, provided that $\inf_k (\Im\lambda_k) >-\infty$.

(ii) The system of rational functions
$\{(1-|\lambda_k|^2)^{1/2}(1-\lambda_k
z)^{-1}\}_{\lambda_k\in\Lambda}$ in $L^2(\T)$.

In the forthcoming paper \cite{gmo2} we shed light on this effect and show that
a transparent connection of the result in Theorem \ref{theorem3.25} with the corresponding
property of the system of exponential functions is not occasional and has deeper reasons.

From the very starting point we were  influenced by the paper
\cite{MalSch12}, wherein a tight connection between the spectral
theory of  $3D$ Schr\"{o}dinger operators with {\it infinitely many}
point interactions  and RPDF's in $\R^3$ was discovered and
exploited in both directions. In particular, a special case of
Theorem \ref{intrinftystrong} (for $n=d=3$ and $\alpha=1$) was
proved in \cite{MalSch12}  by applying machinery of the spectral
theory and the grammization of the Schoenberg--Bernstein matrix $\kS_X(e^{-as})$,
which is achieved for $n=3$  by the system
$$ f_{a,1/2}(x-x_j) = \sqrt{\frac{a}{|x-x_j|}}\,{K_{1/2}(a|x- x_j|)}
=\sqrt{\frac{\pi}{2}}\,\frac{e^{-a|x- x_j|}}{|x- x_j|}\,, \quad j\in\N, $$
(see \eqref{malshmu1}).  However the spectral methods applied in
\cite{MalSch12} \emph{cannot be extended to either $n\ge4$ or
$\alpha\not =1$}. Our reasoning is based on the harmonic and Fourier
analysis on $\R^n$ and works for an arbitrary dimension $n\ge2$.

{\bf Acknowledgement}.   We are grateful to T. Gneiting for comments
on $\alpha$-stable functions and the Mat\'ern classes and A.
Kheifets for the function theoretic argument in the proof of
Lemma~\ref{Toeplitz_Lemma}.

\section{Preliminaries}\label{prelim}

\subsection{Positive definite functions}\label{prelimpdf}

Recall some basic facts and  notions  related to positive  definite functions \cite{Akh65, BCR, TriBel, Wend05}.

\begin{definition}\label{defpoz}
A function $g:\R^n\to\C$ is called {\it positive definite} if $g$ is continuous at the origin and for
arbitrary finite sets $\{x_1,\dots,x_m\}$, $x_k\in \R^n$ and $\{\xi_1,\dots,\xi_m\}\in\C^m$ we have
\begin{equation}\label{positiv}
\sum_{k,j=1}^{m} g(x_k-x_j)\xi_j\overline{\xi}_k\ge 0.
\end{equation}
  \end{definition}
The set of positive definite function on $\R^n$ is denoted by $\Phi(\mathbb{R}^n)$. Clearly, a function $g\in\Phi(\R^n)$
if and only if it  is continuous at the origin, and the  matrix
$\kB_X(g):=\|g(x_k-x_j)\|_{k,j=1}^m$ is nonnegative, $\kB_X(g)\ge0$, for all finite subsets $X=\{x_j\}_{j=1}^m$ in $\R^n$.

A celebrated theorem of S. Bochner \cite{Boch} gives a description of the class $\Phi(\R^n)$.

 \begin{theorem}\label{boch}
A  function $g$  is  positive  definite  on  $\R^n$ if  and only if there exists a finite positive Borel
measure $\mu$ on  $\R^n$ such that
\begin{equation}\label{bochequa}
g(x)=\int\nolimits_{\R^n}e^{i(u,x)}\mu(du), \qquad x\in \R^n.
\end{equation}
    \end{theorem}
When $g$ is a radial function, $g(\cdot)=f(|\cdot|)$, $f\in\Phi_n$, the representing measure $\nu$ in \eqref{schoenberg1}
is related to the Bochner measure $\mu$ by $\nu\{[0,r]\}=\mu\{|x|\le r\}$ (cf. \cite[Section V.4.2]{Akh65}).

\subsubsection{Class $\Phi_\infty$ of radial positive definite functions}\label{prelimpdf}

Going over to the classes $\Phi_n$ of PRDF's, note that the sequence
$\{\Phi_n\}_{n\in\N}$ is known to be nested, i.e.,
$\Phi_{n+1}\subset\Phi_n$, and inclusion is proper (see
\cite{Sch38}, \cite[Section 6.3]{TriBel}). So the intersection
$\Phi_\infty=\bigcap\limits_{n\in\mathbb{N}}\Phi_n$ comes in
naturally. The class $\Phi_\infty$ is the case of study
in the pioneering paper of I. Schoenberg \cite{Sch38}.
According to the Schoenberg theorem (see, e.g., \cite[Theorem
5.4.3]{Akh65}),  $f\in \Phi_\infty$, $f(0)=1$, if and only if it admits  an
integral representation
  \begin{equation}\label{Scoenberg_alpha-2}
f(t) = \int_0^\infty e^{-st^2}\,\sigma(ds), \qquad t\ge0,
  \end{equation}
with $\sigma$ being a probability measure on $\R_+$.
The measure $\sigma$, which is called a {\it Schoenberg measure} of
$f\in \Phi_\infty$, is then uniquely determined by $f$.

Another characterization of the class $\Phi_\infty$  is
$\Phi_\infty=\Phi(\ell^2)$, where  the latter is the class of radial
positive definite functions on the real Hilbert space $\ell^2$ (see,
e.g., \cite[p.283]{TriBel}). Indeed, since $\R^n$ is embedded in
$\ell^2$ for each $n\in\N$, we have
$\Phi(\ell^2)\subset\Phi_\infty$. Conversely, let $f\in\Phi_\infty$
and $Y=\{y_k\}_{k=1}^m\subset\ell^2$, $y_k=(y_{k1}, y_{k2},\ldots)$.
Define truncations $y_k^{(n)}:=(y_{k1},y_{k2}, \ldots, y_{kn},
0,0,\ldots)\in\R^n$. Then for each $n$
$$ \|f(|y_i^{(n)}-y_j^{(n)}|)\|_{i,j=1}^m\ge0. $$
As $\lim_{n\to\infty}|y_i^{(n)}-y_j^{(n)}|=|y_i-y_j|$
and $f$ is continuous, the matrix $\|f(|y_i-y_j|)\|_{i,j=1}^m$
is also positive definite, as claimed.

\subsubsection{Bernstein class $CM(\R_+)$  of absolute monotone functions}\label{Bernstein_prelimpdf}

\begin{definition}
A  function  $f\in C(\R_+)$ is called {\it completely monotone} if
\begin{equation}\label{derivativescm}
(-1)^kf^{(k)}(t)\ge 0, \qquad t>0, \quad k=0,1,2,\ldots.
\end{equation}
The set of  such functions is denoted by $CM(\R_+)$. A function $f$
belongs to a subclass  $CM_0(\R_+)$ of $CM(\R_+)$ if $f\in CM(\R_+)$
and  $f(+0)=1$.
  \end{definition}
A fundamental theorem of S. Bernstein -- D. Widder (\cite{beram, Wid}, see also \cite[p.204]{Akh65})
claims that $f\in CM(\R_+)$ if and only if there exists a positive Borel measure $\tau$ on $\R_+$
such that
   \begin{equation}\label{bernstein}
f(t) = \int^{\infty}_0 e^{-st}\tau(ds), \qquad   t>0.
   \end{equation}
The measure $\tau$, which is called a {\it Bernstein measure} of
$f\in CM(\R_+)$, is then uniquely determined by $f$. $\tau$ is the
probability measure if and only if $f\in CM_0(\R_+)$.
%So $\Phi_\infty(1)=CM_0(\R_+)$.

\subsubsection{Subclasses $\Phi_\infty(\alpha)$ of radial  positive definite functions}\label{Subcl_Alpha_prelimpdf}

By definition,  a class $\Phi_\infty(\alpha)$ consists of functions
 which admit an integral representation
  \begin{equation}\label{alphacm}
f(t)=\int_0^\infty e^{-st^\alpha}\,\sigma(ds), \qquad t\ge0, \quad 0<\alpha\le2,
  \end{equation}
$\sigma$ is a probability measure on $\R_+$. We call the functions
$f\in\Phi_\infty(\alpha)$ {\it $\alpha$-stable}. They are tightly
related to $\alpha$-stable distributions in probability theory.
So, $\Phi_\infty(2)=\Phi_\infty$, $\Phi_\infty(1)=CM_0(\R_+)$.
The classes $\Phi_\infty(\alpha)$ are known to admit the following
characterization \cite{bdk}: $f\in \Phi_\infty(\alpha)$,
$0<\alpha\le2$, if and only if the function $f(|x|_\alpha)$ is
positive definite, where
$$ x=(x_1,x_2,\ldots), \qquad |x|_\alpha:=\left(\sum_{n=1}^\infty |x_j|^{\,\alpha}\right)^{\frac1{\alpha}}. $$
Note that the family $\{\Phi_\infty(\alpha)\}_{0<\alpha\le2}$ is nested, i.e.,
\begin{equation}\label{inclstable}
\Phi_\infty(\alpha_1)\subset \Phi_\infty(\alpha_2), \qquad 0<\alpha_1<\alpha_2\le2,
\end{equation}
and the inclusion is proper (see, e.g., \cite{bdk, gn1998}). Indeed,
\eqref{inclstable} is equivalent to
\begin{equation}\label{inclstable1}
\Phi_\infty(\alpha)\subset \Phi_\infty(1)=CM_0(\R_+), \qquad 0<\alpha<1,
\end{equation}
(a simple change of variables under the integral sign). Next, it is known (and can be easily verified by induction,
using Leibniz chain rule) that the function $f=e^{-g}\in CM(\R_+)$ provided $g'\in CM(\R_+)$. Hence
$$ \exp(-sx^\alpha)\in CM_0(\R_+), \qquad 0<\alpha\le1, $$
so \eqref{derivativescm} holds for this function. Differentiation under
the integral sign shows that the same is true for each $f\in\Phi_\infty(\alpha)$ and \eqref{inclstable1} follows.
The same argument implies $\exp(-sx^\beta)\notin \Phi_\infty(\alpha)$ for $\beta>\alpha$.

For the detailed account of the subject see, e.g., \cite[Chapter 2.7]{WelWil}.

\subsection{Infinite matrices and Schur test}

We say that an infinite matrix $\kA=\|a_{kj}\|_{k,j\in \N}$ with complex entries $a_{kj}$ generates a bounded linear
operator $A$ on the  Hilbert space $\ell^2=\ell^2(\N)$ (or simply that an infinite matrix is a bounded operator on $\ell^2$)
if there exists a bounded linear operator $A$ such that
  \begin{equation}\label{matrixoperator}
\langle Ax,y\rangle =\sum_{k,j=1}^\infty
a_{kj}x_k\overline{y_j}, \qquad   x=\{x_k\}_{k\in \N},\quad y=\{y_k\}_{k\in \N}, \quad x,y\in \ell^2.
  \end{equation}
Clearly, if $\kA$ defines a bounded operator $A$, then $A$ is
uniquely determined by equalities (\ref{matrixoperator}).

The following result known as the {\it Schur test} (due in substance to I. Schur)
provides certain general conditions for an infinite matrix $\kA=\|a_{ij}\|_{i,j\in\N}$ to define a bounded linear operator
$A$ on $\ell^2$ (see, e.g., \cite[Theorem 5.2.1]{Nik1}). One of the simplest its versions can  be stated as follows.
   \begin{lemma}\label{schurtest}
Let $\kA=\|a_{ij}\|_{i,j\in \N}$ be an infinite Hermitian matrix which satisfies
    \begin{equation}\label{schurtestas}
C:= \sup_{j\in \N}\,\sum_{i=1}^\infty |a_{ij}| <\infty.
   \end{equation}
Then $\kA$ defines a  bounded self-adjoint operator $A$ on $\ell^2$ with $\|A\|\leq C$.
     \end{lemma}
Note that the Schur test applies to general (not necessarily Hermitian) matrices with two independent
conditions for their rows and columns
    \begin{equation*}
C_1:= {\sup}_{j\in \N}\,\sum_{i=1}^\infty |a_{ij}| <\infty, \quad C_2:= {\sup}_{i\in \N}\,\sum_{j=1}^\infty |a_{ij}| <\infty,
   \end{equation*}
and the bound for the norm is $\|A\|^2\le C_1C_2$.

The condition for compactness of $A$ is similar.
 \begin{lemma}\label{schurtest1}
Suppose that
\begin{equation}\label{schurtestas1}
\delta_p:=\sup_{j\geq p}\,\sum_{k\geq p}|a_{jk}|<\infty, \quad \forall p\in\N, \quad {\rm and} \quad \lim_{p\to \infty} \delta_p=0.
   \end{equation}
Then the Hermitian matrix $\kA=\|a_{kj}\|_{k,j\in \N}$ generates a compact self-adjoint operator on $\ell^2$.
     \end{lemma}

For the proof see, e.g., \cite[Lemma 2.23]{MalSch12}

\section{Schoenberg matrices from operator theory viewpoint}\label{shoenoper}

\subsection{Bounded Schoenberg operators}

Sometimes an infinite Schoenberg matrix generates a bounded linear operator $S_X(f)$
on $\ell^2$. We call $S_X(f)$ a {\it Schoenberg operator}. The
main problem we address here concerns conditions on the test set $X\subset\R^n$
and the Schoenberg symbol $f$ for $S_X(f)$ to be bounded.

We will be dealing primarily with separated sets $X$,
\begin{equation*}%\label{separat}
d_*=d_*(X):=\inf_{i\not=j}|x_i-x_j|>0.
\end{equation*}
Recall the notation $\kX_n$ for the class of all separated sets in $|R^n$ and $\kL=\kL(X)$ for the linear span of $X$,
$d=\dim\kL\le n$.

The result below gives an upper bound for the number of points of a separated set $X$ in a spherical layer
     \begin{equation*}
U_{r}(p,q,a,X):=\{y\in\kL(X):\ pr \le |y-a|< qr\}, \quad q>p\ge0,
     \end{equation*}
centered at $a\in \kL(X)$.

\begin{lemma}\label{layer}
Let $X=\{x_k\}_{k\in\N}\in\kX_n$, $d_\ast(X)=\varepsilon>0$, and let $a\in \kL(X)$.
Then for the number $N_m(X)$ of the points $\{x_j\}$ contained in $U_{\varepsilon}(m,m+1,a,X)$, $m=0,1,\ldots$, the inequality
\begin{equation}\label{numblayer}
N_m(X)=\bigl|X\bigcap U_{\varepsilon}(m,m+1,a,X)\bigr|\le (2m+3)^d-(2m-1)^d<d\,5^d\,m^{d-1}
\end{equation}
holds.
\end{lemma}
\begin{proof}
Take $x_j\in X\cap \Bbb U_{\varepsilon}(m,m+1,a,X)$ and consider the balls
$B_{\varepsilon/2}(x_j)=\{x\in\kL: |x-x_j|<\varepsilon/2\}$, centered at $x_j$. They are contained in the spherical layer
$U_{\varepsilon}(m-1/2,m+3/2,a,X)$,
and pairwise disjoint. Since the volume of this layer is
$$ |U_{\varepsilon}(m-1/2,m+3/2,a,X)|=\kappa_d\Bigl[\bigl((m+3/2)\varepsilon\bigr)^d - \bigl((m - 1/2)\varepsilon\bigr)^d\Bigr],
   \qquad \kappa_d=\frac{\pi^{d/2}}{\Gamma\left(\frac{d}2+1\right)} $$
is the volume of the unit ball in $\R^d$, and $|B_{\varepsilon/2}(x_j)|=\kappa_d(\varepsilon/2)^d$, the number $N_m(X)$
satisfies \eqref{numblayer}, as claimed.
   \end{proof}

\smallskip

As far as the Schoenberg symbol $f$ in the definition of
Schoenberg's matrices goes, we assume here that it is a nonnegative,
monotone decreasing function on $\R_+$, and $f(0)=1$, i.e. $f\in
\kM_+$ (cf. \eqref{function}).  Further assumptions on the behavior
of $f$ at infinity will vary.

We proceed with a simple technical result.

\begin{lemma}\label{techlemma}
Let $h\in\kM_+$ and $d\in\N$. Then
\begin{equation}\label{decay1}
\sum_{m=1}^\infty m^{d-1}\,h(m)<\infty \ \ \Longleftrightarrow \ \ \int_0^\infty t^{d-1}\,h(t)\,dt<\infty.
\end{equation}
More precisely, for all $p\in\N$
\begin{equation}\label{twosidebound}
2^{-d+1}\int_p^\infty t^{d-1}\,h(t)\,dt\le \sum_{m=p}^\infty m^{d-1}\,h(m)\le d\int_{p-1}^\infty t^{d-1}\,h(t)\,dt.
\end{equation}
\end{lemma}
\begin{proof}
An elementary inequality
$$ \frac{m^{d-1}}{d}\le \frac{m^d-(m-1)^d}{d}\le m^{d-1}, \qquad m\in\N, $$
gives for $h\in\kM_+$
$$ \int_{m-1}^m t^{d-1}\,h(t)\,dt \ge h(m)\int_{m-1}^m t^{d-1}\,dt=h(m)\,\frac{m^d-(m-1)^d}{d}\ge \frac{m^{d-1}h(m)}{d}, $$
so summation over $m$ leads to the right inequality in \eqref{twosidebound}. Similarly,
$$ \int_{m}^{m+1} t^{d-1}\,h(t)\,dt\le h(m)\,\int_{m}^{m+1} t^{d-1}\,dt=h(m)\,\frac{(m+1)^d-m^d}{d}\le (m+1)^{d-1}h(m), $$
and hence
$$ \sum_{m=p}^\infty (m+1)^{d-1}\,h(m)\ge \int_p^\infty t^{d-1}\,h(t)\,dt. $$
It remains only to note that $m+1\le 2m$ for $m\in\N$.
\end{proof}

For a one dimensional $X$, i.e., $d(X)=1$, condition \eqref{decay1} is just $f\in L^1(\R_+)$.

Recall that we write $X\in\kX_d$, $d\le n$, if $X\in\kX_n$ and $\dim\kL(X)=d$.

The following notion will be crucial in the second part of Theorem \ref{bochbound} below.

\begin{definition}\label{regularset}
A set $Y=\{y_j\}_{j\in\N}\in\kX_d$
is called {\it $\delta$-regular} if there are constants $c_0=c_0(d,\delta,Y)>~0$ and
$r_0=r_0(d,Y)\ge0$, independent from $j$ such that
\begin{equation}\label{regseparat}
|Y^{(j)}_r(\delta)|\ge c_0(d,\delta,Y)\,r^{d-1}, \qquad Y_r^{(j)}(\delta):=\{y_k\in Y: \ r\le |y_k-y_j|<r+\delta\},
\end{equation}
for $r\ge r_0$ and $j\in\N$.
\end{definition}
For instance, the lattice $\Z^n$ and its part $\Z_+^n$ are $\delta$-regular for all $\delta>0$.
On the other hand, if $X=\{x_k\}_{k\in\N}\in\R^n$, $\kL(X)=\R^n$ but $X^{(p)}:=\{x_k\}_{k\ge p}\subset\R^{n-1}$
then $X$ is certainly irregular.

Note that for any regular set $Y$ the number $N_r^{(j)}$ of points in the set $Y\cap \{y:|y-y_j|\le r\}$ is subject to the bounds
\begin{equation}\label{pointsinball}
c_1r^d\le N_r^{(j)}\le c_2 r^d
\end{equation}
for all large enough $r$. Here and in the proof of Theorem \ref{bochbound} $c_k$ stand for different positive
constants which depend on $d, \delta$, and $Y$.

\begin{theorem}[=Theorem \ref{intrbochbound}]\label{bochbound}
Let $f\in\kM_+$, $X\in\kX_n$ and let $d = \dim\kL(X)$.
\begin{itemize}
  \item [\em (i)]
 If $t^{d-1}f(\cdot)\in L^1(\R_+)$, then the Schoenberg matrix  $\kS_X(f)$ is bounded
on $\ell^2$ and
\begin{equation}
\|S_X(f)\|\le 1+d^2\biggl(\frac{5}{d_*(X)}\biggr)^d\,\int_0^\infty
t^{d-1}\,f(t)\,dt.
\end{equation}
\item [\em (ii)]
Moreover, $S_X(f)$ has a bounded inverse  whenever, in addition,
\begin{equation}\label{boundforinverse}
d_*(X)>5d^{2/d}\,\|t^{d-1}f\|_{L^1(\R_+)}^{1/d}.
\end{equation}
  \item [\em (iii)]
Conversely, let $S_Y(f)$ be bounded for at least one
$\delta$-regular set $Y$. Then $t^{d-1}f(\cdot)\in L^1(\R_+)$.
\end{itemize}
\end{theorem}

\begin{proof}
(i). We apply the Schur test to $\kS_X(f)=\|f(|x_k-x_j|)\|_{k,j\in\N}$. For a fixed $j\in\N$ and
$\varepsilon=d_*(X)>0$ denote
     \begin{equation}\label{setinlayer}
X^{(j)}_m := \{x_k\in X:  m\varepsilon \le|x_k-x_j|<
(m+1)\varepsilon\}, \quad m\in\N, \quad X^{(j)}_0=\{x_j\}.
     \end{equation}
By Lemma \ref{layer} $|X^{(j)}_m|<d\,5^d\,m^{d-1}$. Combining this estimate with
the  monotonicity of $f$ yields
\begin{equation}\label{2.13}
\begin{split}
\sum_{k=1}^\infty f(|x_k-x_j|) &= 1+\sum_{m=1}^\infty\,\sum_{x_k\in X^{(j)}_m}\,f(|x_k-x_j|)
\le 1+\sum_{m=1}^\infty |X^{(j)}_m|\,f(m\varepsilon)  \\
&\le 1+d5^d\sum_{m=1}^\infty m^{d-1}\,f(m\varepsilon).
\end{split}
\end{equation}
The result now follows from the Schur test and Lemma \ref{techlemma} with $h(\cdot)=f(\varepsilon \cdot)$.

(ii). Going over to the second statement, one has as above
$$ \sum_{k=1}^\infty |f(|x_k-x_j|)-\delta_{kj}|=\sum_{k\not=j}f(|x_k-x_j|) \le d^2\biggl(\frac{5}{d_*(X)}\biggr)^d\,\int_0^\infty t^{d-1}\,f(t)\,dt,
$$
so $\|S_X(f)-I\|<1$ as soon as \eqref{boundforinverse} holds and $S_X(f)$ is invertible.

(iii). With no loss of generality assume that $\kL(X)=\R^d$. At this point we make use of a particular labeling of the set $X$ (generally speaking
the way of enumeration of $X$ makes no difference in our setting). Precisely, we label $X$ by increasing of the distance from the origin
$$ 0=|x_1|<|x_2|\le |x_3|\le\ldots. $$

For a ball $B_r=B^d_r$ of radius $r>0$ centered at the origin we put $E_r:=X\cap B_r$ and $N_r:=|E_r|$. Given $x_j\in X$, denote by $p(j)$ the number
of layers $X^{(j)}_m$ which are contained in $B_r$. It is clear that for any $x_j\in E_{r/2}$ one has $p(j)\ge [r/2\varepsilon]$.

From the Definition \ref{regularset} and $f\in\kM_+$ we see that
\begin{equation}\label{belowbound1}
\begin{split}
\sum_{k=1}^{N_r} f(|x_k-x_j|) &\ge \sum_{m=1}^{p(j)} \sum_{x_k\in X^{(j)}_m} f(|x_k-x_j|)\ge c_3\sum_{m=1}^{p(j)} m^{d-1} f(\varepsilon(m+1)) \\
&\ge c_4\sum_{m=2}^{p(j)+1} m^{d-1} f(\varepsilon m).
\end{split}
\end{equation}

Since $S_X(f)$ is bounded then on the test vector $h=h_{N_r}=\frac1{\sqrt{N_r}}(1,1,\ldots,1,0,0,\ldots)$, $\|h\|=1$, we have in view of
\eqref{belowbound1} and  \eqref{pointsinball} (with $j=1$, $x_1=0$)
\begin{equation*}
\begin{split}
\|S_X(f)\| &\ge |\langle S_X(f)h,h\rangle|=\frac1{N_r}\sum_{j=1}^{N_r}\sum_{k=1}^{N_r} f(|x_k-x_j|)
\ge \frac1{N_r}\sum_{|x_j|<R/2}^{N_r}\sum_{k=1}^{N_r} f(|x_k-x_j|) \\
&\ge \frac{c_5}{N_r}\,N_{r/2}\sum_{m=2}^{[r/2\varepsilon]} m^{d-1} f(\varepsilon m)
\ge c_6 \sum_{m=2}^{[r/2\varepsilon]} m^{d-1} f(\varepsilon m).
\end{split}
\end{equation*}
Since $r$ is arbitrarily large, the result follows from Lemma \ref{techlemma}.
\end{proof}

\begin{remark}
The statement (iii) of the above theorem is particularly simple for $d=1$.

A one-dimensional sequence $X(\La)=\{x_k\}$, $x_k=\la_k e$, is called a Toeplitz-like sequence if
\begin{equation}\label{toeplitz-like}
 0=\la_1<\la_2<\ldots, \qquad 0<d_*(X)\le \la_{i+1}-\la_i\le d^*(X)<\infty, \end{equation}
for all $i\in\N$.

Assume now that the Schoenberg operator $S_X(f)$ is bounded. Take the same test vector
$h_N=\frac1{\sqrt{N}}(1,1,\ldots,1,0,0,\ldots)$, $\|h_N\|=1$ and write
$$ \|S_X(f)\|\ge\langle S_X(f)h_N,h_N\rangle=\frac1{N}\sum_{i,j=1}^N f(|x_i-x_j|)=
f(0)+\frac2{N}\sum_{k=1}^{N-1}\sum_{i=1}^{N-k}f(\la_{i+k}-\la_i), $$
By \eqref{toeplitz-like}, $kd_*(X)\le \la_{i+k}-\la_i\le kd^*(X)$, and in view of monotonicity
$$ \|S\|\ge 2\sum_{k=1}^{m-1} \Bigl(1-\frac{k}{m}\Bigr)f(kd^*(X))\ge 2\sum_{k=1}^{m/2} \Bigl(1-\frac{k}{m}\Bigr)f(kd^*(X))
\ge \sum_{k=1}^{m/2} f(kd^*(X)). $$
Thereby the series $\sum_k f(kd^*(X))$ converges and Lemma \ref{techlemma} gives $f\in\L^1(\R_+)$.
\end{remark}

For $\alpha$-stable functions we have a simple condition for the boundedness of $S_X(f)$ in terms of the Schoenberg measure
$\sigma$ \eqref{alphacm}.

\begin{corollary}\label{alphaclass}
Let $f\in\Phi_\infty(\alpha)$, $0<\alpha\le2$,  $d\in \N$, and let
$\sigma = \sigma_f$ be the Schoenberg measure in $\eqref{alphacm}$. Then
  \begin{equation}\label{alphabound}
\int_0^\infty t^{d-1}f(t)\,dt<\infty \ \ \Longleftrightarrow \ \ \int_0^\infty s^{-\frac{d}{\alpha}}\,\sigma(ds)<\infty.
 \end{equation}
In particular, the Schoenberg operator $S_X(f)$ is bounded for all $X\in\kX_d$, provided that the measure $\sigma$ satisfies
\eqref{alphabound}.
\end{corollary}
%%%%%%%%%%%%%%%%%%%%%%%%%
\begin{proof} It is clear that $f\in\kM_+$. Next,
\begin{equation}\label{2.18C}
\begin{split}
\int_0^\infty t^{d-1}f(t)\,dt &=\int_0^\infty t^{d-1}\,dt\,\int_0^\infty e^{-st^\alpha}\,\sigma(ds)=
\int_0^\infty \sigma(ds)\,\int_0^\infty t^{d-1}e^{-st^\alpha}\,dt \\ &=\frac1{\alpha}\,\Gamma\left(\frac{d}{\alpha}\right)\,
\int_0^\infty s^{-\frac{d}{\alpha}}\,\sigma(ds)<\infty.
  \end{split}
\end{equation}
Theorem \ref{bochbound} completes the proof. \end{proof}

Note that the above argument goes through for an arbitrary $\alpha>0$.

We prove later in Theorem \ref{inftystrong} that each Schoenberg operator  $S_X(f)$  with the
symbol as in Corollary \ref{alphaclass} is actually invertible.

As another direct consequence of Theorem \ref{bochbound} we have

\begin{corollary}
Let $f,g\in\kM_+$ and $f(t)=g(t)$ for $t\ge t_0$. If $S_Y(f)$ is bounded for at least one regular set $Y\in\kX_d$, then $S_X(g)$ are
bounded for all $X\in\kX_d$.
\end{corollary}

The monotonicity condition in \eqref{function} is somewhat restrictive. It is not at all necessary for Schoenberg's operator to be bounded.

\begin{proposition}\label{nonmonoton}
Let $f$ and $h$ be real-valued functions on $\R_+$. Assume that $|f|\le h$ and the operator $S_X(h)$ is bounded. Then so is
$S_X(f)$. In particular, let $f$ be a bounded function on $\R_+$, which is monotone decreasing for $t\ge t_0(f)$ and
$t^{d-1}f(\cdot)\in L^1(\R_+)$. Then $S_X(f)$ is bounded.
   \end{proposition}
%%%%%%%%%%%%%%%%%%%%%%%
\begin{proof}
The Schoenberg matrix $\kS_X(h)$ dominates $\kS_X(f)$, i.e., $h(|x_j-x_k|)\ge |f(|x_j-x_k|)|$.
Hence if $S_X(h)$ is bounded then by \cite[Theorem 29.2]{AkhGlz78}, so is $S_X(f)$.

Concerning the second statement, it is clear that $f\ge0$ on $[t_0(f),\infty)$. Put $h(t):=\sup_{t\ge s} |f(s)|$. Then $h$ is a
nonnegative function, monotone decreasing on $\R_+$, $h(0)>0$ (we assume $f\not\equiv 0$), and $h=f$ on $[t_0(f),\infty)$, so \eqref{decay1}
holds for $h$. By Theorem \ref{bochbound}, $S_X(h)$ is bounded and as $h\ge |f|$ on $\R_+$, then by the first part of the proof,
so is $S_X(f)$, as needed.
\end{proof}
\begin{corollary}\label{cor2.11}
Let $g\in\Phi_n$,  $\alpha>0$, and $e_{\alpha}(t):=e^{-\alpha t}$.
 Then $f_{\alpha}:= e_{\alpha}g\in\Phi_n$ and for any $d\in\Bbb N$ and any
$X\in \mathcal X_d$ the Schoenberg operator $S_X(f_{\alpha})$ is bounded.
   \end{corollary}
%%%%%%%%%%%%%%%%%%%%%%%%%%%%%%%%%%%%%%%%%%%%%%%%%%%%
       \begin{proof}
Since $e_{\alpha}\in CM_0(\R_+)\subset\Phi_{\infty}$, then for any finite $X$ the  Schoenberg
matrix  $S_X(f_{\alpha}) = S_X(e_{\alpha})\circ S_X(g)$, being  the
Schur product of two non-negative matrices $S_X(e_{\alpha})$ and
$S_X(g)$, is also non-negative. This proves the inclusion $f_{\alpha}\in\Phi_n$.

Next, since $|f_{\alpha}(t)|\le M e^{-\alpha t}$ with
$M=\|g\|_{C(\Bbb R_+)}$, then $t^{d-1}f_\alpha(\cdot)\in L^1(\R_+)$
with an arbitrary $d\in\N$. It remains to apply Proposition  \ref{nonmonoton}.
       \end{proof}

\subsection{Fredholm property}

We discuss here the situation when $S_X(f)$ is a Fredholm operator, more precisely,
\begin{equation}\label{onepluscomp}
S_X(f)=I+T, \qquad T\in \frak S_\infty(\ell^2)
\end{equation}
is a compact operator on $\ell^2$. In this case one should
impose a much stronger condition on $X$ than just $d_*(X)>0$.
\begin{theorem}\label{bochcomp}
Let $X=\{x_k\}_{k\in\N}\subset\R^d$ satisfy
\begin{equation}\label{strongseparat}
\lim_{{i,j\to\infty}\atop{i\not=j}} |x_i-x_j|=+\infty.
\end{equation}
Let $f\in\kM_+$ with $t^{d-1}f\in L^1(\R_+)$.
Then $\eqref{onepluscomp}$ holds. In particular,  $S_X(f)$ has bounded inverse whenever $\ker
S_X(f)=\{0\}$.

Conversely, let $f$ be a strictly positive, monotone decreasing function on $\R_+$, $f(0)=1$, and $t^{d-1}f\in L^1(\R_+)$.
Then $\eqref{onepluscomp}$ implies $\eqref{strongseparat}$.
\end{theorem}
\begin{proof}
To apply Lemma \ref{schurtest1} we argue as in the proof of the
Theorem \ref{bochbound}. According to  Lemma \ref{layer}  for each
$p\in\N$ there is $q=q(p)\in\N$ so that for $j\ge p$
\begin{equation*}
\begin{split}
\sum_{k=p}^\infty |f(|x_k-x_j|)-\delta_{kj}| &= \sum_{k\ge p,\, k\not=j}f(|x_k-x_j|)= \sum_{m=q}^\infty\,\sum_{x_k\in X^{(j)}_m}\,f(|x_k-x_j|)
\le d5^d\sum_{m=q}^\infty m^{d-1}\,f(d_*(X)m) \\ &\le d^2\biggl(\frac{5}{d_*(X)}\biggr)^d\,\int_{d_*(q-1)}^\infty t^{d-1}\,f(t)\,dt.
\end{split}
\end{equation*}
Condition \eqref{strongseparat} implies $q(p)\to\infty$ as
$p\to\infty$ and so operator $T=S_X(f)-I$ is compact by Lemma
\ref{schurtest1}.

Conversely, suppose that there are two sequences $\{i_m\}$, $\{j_m\}$ so that $i_m\not= j_m$, both tend to infinity as $m\to\infty$
and $\sup_m|x_{i_m}-x_{j_m}|\le C<\infty$. Then
$$
0<f(C)\le f(|x_{i_m}-x_{j_m}|)=\langle S_X(f)e_{j_m}, e_{i_m}\rangle=\langle Te_{j_m}, e_{i_m}\rangle,
$$
which contradicts the compactness of $T$. The proof is complete.
\end{proof}
%%%%%%%%%%%%%%%%%%%%%%%%%%%%%%%%%%%%%%%%%%%%%%%%%%%%%%%%%%%%%%%%%%

\begin{example}\label{nontrivkernel}
We show that in the converse statement of Theorem \ref{bochcomp} the condition $f>0$ cannot be relaxed to $f\ge0$.
Take the truncated power function
\begin{equation*}
f(t)=(1-t)_+^l, \qquad l>0.
\end{equation*}
It is known \cite{gol81, zas2000} that $f\in\Phi_n$ if and only if $l\ge\frac{n+1}2$. As a test sequence $X=\{x_k\}_{k\in\N}$ we put
$x_k=a_k\xi$, $\|\xi\|=1$, with
$$ a_1=0, \quad a_2=\frac12, \quad  a_k=k, \quad k=3,4,\ldots, $$
so that
$f(|x_2-x_1|)=2^{-l}$, $f(|x_i-x_j|)=0$ for the rest of the pairs $j\not=i$. The Schoenberg operator now takes the form
$$ S_X(f)= \begin{bmatrix}
 A&{}\\ {}&I
    \end{bmatrix}, \qquad
    A=\begin{bmatrix}
    1& 2^{-l} \\
    2^{-l} & 1
    \end{bmatrix} $$
and $I$ is a unit matrix. It is clear that $S_X(f)=I+T$, $rk\,T=2$, but \eqref{strongseparat} is false.
\end{example}

\subsection{Unbounded Schoenberg operators}

Conditions on an infinite matrix $\kA$ for the corresponding linear operator $A$ on $\ell^2$ to be bounded are rather
stringent. These conditions fail to hold for a number of Schoenberg's matrices (cf. Example \ref{nonell2}).

To broaden the area of our study, consider an infinite Hermitian
matrix $\kA=\|a_{kj}\|_{k,j\in\N}$, $a_{jk}=\bar a_{kj}$, satisfying the
following  conditions
\begin{equation}\label{unboundher}
\sum_{k=1}^\infty |a_{kj}|^2<\infty, \qquad \forall j\in\N.
\end{equation}
Such matrix defines in a natural way a linear operator $A'$ on $\ell^2$
which act on the standard basis vectors $\{e_k\}_{k\in\N}$, $(e_k)_m =
\delta_{km}$, as
$$
A'e_j=\sum_{k=1}^\infty a_{kj}e_k, \qquad j\in\N,
$$
extended by linearity to the linear span $\kL$ of $\{e_k\}_{k\in\N}$, so $A'$
is densely defined and $\dom(A')\supset\kL$. Being symmetric (since
$\kA$ is a Hermitian matrix), $A'$ is closable, and we denote by $A
= \overline {A'}$ its closure. The operator $A$ is called a minimal
operator associated with $\kA$.

Matrices \eqref{unboundher} are usually referred to as {\it unbounded Hermitian matrices} (unless $A$ is a bounded operator).

A maximal operator associated with such matrix $\mathcal A$ is given by
     \begin{equation}
A_{\max}f = \sum^{\infty}_{k=1}b_k e_k, \qquad b_k =\sum^{\infty}_{k=1}a_{kj}x_j,
     \end{equation}
on the domain
    \begin{equation*}
\dom\,(A_{\max})=\left\{f=\sum^{\infty}_{k=1}x_k e_k:\ \sum^{\infty}_{k=1}
\left|\sum^{\infty}_{j=1}a_{kj}x_j \right|^2 < \infty\right\}.
    \end{equation*}
It is known (see, e.g., \cite[Theorem 53.2]{AkhGlz78}) that $A_{\max}=A^*$.

Conversely, given a closed symmetric operator $A$ on a Hilbert space $\kH$, an orthonormal basis $\{h_k\}_{k\in\N}$ is called a
basis of the matrix representation of $A$ if
\begin{itemize}
\item
$h_k\in\dom(A)$,  $k\in\Bbb N$;

\item  $A$ is a minimal closed operator sending $h_k$ to $A h_k$,
$k\in\Bbb N$.
\end{itemize}

\medskip

A curious property of certain Schoenberg's matrices is that the validity of \eqref{unboundher} for at least one value of $j$ implies
relation \eqref{unboundher} to hold for all $j\in\N$. We begin with the technical lemma.

Let us say that a finite positive Borel measure $\sigma$ on $\R_+$ possesses a doubling property if there is $\kappa>0$ so that
\begin{equation}\label{doubl}
\sigma[2u,2v]\le\kappa\,\sigma[u,v], \qquad \forall [u,v]\subset\R_+.
\end{equation}

\begin{lemma}\label{curious}
Let $f\in CM_0(\R_+)$ and $\xi, \eta\in\Bbb R^n$. Then there is a constant $C=C(f,\xi,\eta)>0$ such that
      \begin{equation}\label{cureq1}
f(|x-\xi|) < C f(|x-\eta|), \qquad  \forall x\in\Bbb R^n.
      \end{equation}
The same conclusion is true for $f\in\Phi_\infty=\Phi_\infty(2)$ as long as its Schoenberg measure $\sigma$ \eqref{alphacm}
possesses the doubling property.
\end{lemma}

\begin{proof}
First, let $f\in CM_0(\R_+)$. Choose $a=a_f>0$ so that
\begin{equation}\label{cureq2}
\int_0^a \tau(ds)>\frac12 \Rightarrow \int_a^\infty \tau(ds)<\frac12.
\end{equation}
We show that \eqref{cureq1} actually holds with $C=2e^{a|\xi-\eta|}$. Consider two cases.

\noindent
1. Let first $|x-\eta|\le |\xi-\eta|$. Then since $f\le1$, one has
\begin{equation*}
f(|x-\eta|) =\int_0^\infty e^{-s|x-\eta|}\,\tau(ds)\ge \int_0^a e^{-s|x-\eta|}\,\tau(ds)>\frac12\,e^{-a|x-\eta|} \\
\ge \frac12\,e^{-a|\xi-\eta|}\,f(|x-\xi|),
\end{equation*}
as needed.

\noindent
2. Let now $|x-\eta|>|\xi-\eta|$, so $|x-\xi|\ge |x-\eta|-|\xi-\eta|>0$. The function $f$ is certainly monotone decreasing, so
\begin{equation*}
\begin{split}
f(|x-\xi|) &\le f(|x-\eta|-|\xi-\eta|)=\int_0^\infty \exp(-s|x-\eta|+s|\xi-\eta|)\tau(ds) \\
&=\left\{\int_0^a+\int_a^\infty\right\}\,\exp(-s|x-\eta|+s|\xi-\eta|)\tau(ds)=I_1+I_2.
\end{split}
\end{equation*}
Obviously, for every nonnegative and monotone decreasing function $u$ on $\R_+$, condition \eqref{cureq2} implies
$$ \int_0^a u(s)\tau(ds)\ge \frac{u(a)}2>u(a)\,\int_a^\infty \tau(ds)\ge\int_a^\infty u(s)\tau(ds). $$
Hence $I_2\le I_1$. To bound $I_1$ note that
$$ I_1\le e^{a|\xi-\eta|}\,\int_a^\infty e^{-s|x-\eta|}\,\tau(ds)=e^{a|\xi-\eta|}\,f(|x-\eta|), $$
and \eqref{cureq1} follows.

Concerning functions $f\in\Phi_\infty$, the reasoning is identical (with the obvious replacement of $\tau$ with $\sigma$)
up to the bound of $I_1$, where the doubling property comes into play. We now have
$$ I_1=\int_0^a\,\exp(-s(|x-\eta|-|\xi-\eta|)^2)\,\sigma(ds)\le e^{a|\xi-\eta|^2}\int_0^a\,e^{-\frac{s}2|x-\eta|^2}\,\sigma(ds)
      \le e^{a|\xi-\eta|^2} f\biggl(\frac{|x-\eta|}{\sqrt2}\biggr). $$
It remains only to note that
$$ f\biggl(\frac{r}{\sqrt2}\biggr)=\int_0^\infty\,e^{-\frac{s}2 r^2}\,\sigma(ds)\le
\kappa \int_0^\infty\,e^{-sr^2}\,\sigma(ds)=\kappa f(r), \quad r>0 $$
because of the doubling property \eqref{doubl}. The proof is complete.
\end{proof}
%%%%%%%%%%%%%%%%%%%%%%%%%%%%%

\begin{proposition}
Let $f\in CM_0(\Bbb R_+)$, $X=\{x_k\}_{k\in\N}\subset\R^n$, and let $\kS_X(f)$ be the corresponding Schoenberg matrix.
If at least one column of $\kS$ belongs to
$\ell^2$, then $\eqref{unboundher}$ holds and $\{e_k\}_{k\in\N}$ is a basis of the matrix representation for the
minimal operator $A$ associated with $\kS_X(f)$. The same conclusion is true for $f\in\Phi_\infty$ as long as its Schoenberg
measure $\sigma$ possesses the doubling property.
\end{proposition}

\begin{proof}
Let the first column of $\kS$ belong to $\ell^2$. By Lemma \ref{curious} one has
    \begin{equation}\label{2.2New}
\sum^{\infty}_{j=1}|f(x_{j}-x_{k})|^2 \le C^2 \sum^{\infty}_{j=1}|f(x_{j}- x_{1})|^2<\infty
    \end{equation}
for each $k=2,3,\ldots$. The statement about the basis of the matrix representation is obvious.
  \end{proof}

\begin{remark}
It is easy to see that in general for functions off $CM_0(\R_+)$ the doubling property \eqref{doubl} for $\sigma$ cannot be dropped.

Put
$$ a_n:=\sqrt{\log n+2\log\log n}, \qquad n\ge2. $$
Then clearly
$$\sum_{n=2}^\infty e^{-a_n^2}=\sum_{n=2}^\infty \frac1{n\log^2 n}<\infty, \quad
\sum_{n=2}^\infty e^{-(a_n-1)^2}=\frac1{e}\sum_{n=2}^\infty \frac{e^{2a_n}}{n\log^2 n}=\infty. $$

Consider now the Schoenberg matrix $\kS_X(f)$ with
$$ f(t)=e^{-t^2}\in\Phi_\infty\backslash CM_0(\R_+), \qquad X=\{x_k\}_{k\in\N}\subset \R^1: \ \ x_1=0, \ x_2=1, \ x_n=a_n, \quad n\ge3. $$
Then
$$ \sum_{n=1}^\infty f^2(|x_n-x_1|)<\infty, \qquad \sum_{n=1}^\infty f^2(|x_n-x_2|)=\infty. $$
Certainly, now $\sigma=\delta\{1\}$ has no doubling property. Note that in this instance the conclusion of Lemma \ref{curious} is false either.
\end{remark}

In the above example the set $X$ is not separated, that is, $d_*(X)=0$. As we will see later in
Theorem \ref{inftystrong}, the Schoenberg operator $S_X(e^{-t^2})$ is bounded and invertible whenever $d_*(X)>0$,
so all columns belong to $\ell^2$.

\medskip

There is an intermediate condition on the Schoenberg matrix $\kS_X(f)$ between \eqref{unboundher} and the boundedness. Precisely,
\begin{equation}\label{maxdomain}
 \sup_j \sum_{k=1}^\infty f^2(|x_k-x_j|)=C(f,X)<\infty.
\end{equation}
In other words, $\sup_j\|S_X(f)e_j\|<\infty$.

Recall that $\delta$-regular sets are defined in Definition \ref{regularset} above.
\begin{proposition}\label{maxoperator}
Let $f\in\kM_+$, that is, $f$ enjoys condition $\eqref{function}$, and
\begin{equation}\label{decay2}
\int_0^\infty t^{d-1}f^2(t)\,dt<\infty.
\end{equation}
Then $\eqref{maxdomain}$ holds for each separated set $X\in\kX_d$. Conversely, assume that
\begin{equation}\label{maxdomain1}
\sum_{k=1}^\infty f^2(|y_k-y_j|)=C(f,Y)<\infty
\end{equation}
for some $j\in\N$ and at least one $\delta$-regular set $Y$. Then $\eqref{decay2}$ holds with $d=\dim\kL(Y)$.
\end{proposition}
\begin{proof}
Let \eqref{decay2} hold. We apply Lemma \ref{techlemma} with $h=f^2$ and obtain as above
$$ \sum_{k=1}^\infty f^2(|x_k-x_j|)\le 1+\sum_{m=1}^\infty |X_m^{(j)}|\,f^2(d_*(X)m)
\le 1+d^2\biggl(\frac{5}{d_*(X)}\biggr)^d\,\int_{0}^\infty s^{d-1}f^2(s)ds, $$
so \eqref{maxdomain} follows.

Conversely, let $f$ satisfy \eqref{maxdomain1} for some $j\in\N$ and
some $\delta$-regular set $Y=\{y_j\}_{j\in\N}$. In view of the lower bound \eqref{twosidebound} one has by Lemma
\ref{techlemma},
\begin{equation*}
\begin{split}
\sum_{k=1}^\infty f^2(|y_k-y_j|) &=1+\sum_{m=1}^\infty\,\sum_{y_k\in Y^{(j)}_m}\,f^2(|y_k-y_j|)
\ge 1+c_2(d)\sum_{m=1}^\infty m^{d-1}\,f^2(d_*(Y)(m+1)) \\
&\ge 1+c_3(d)\sum_{m=2}^\infty m^{d-1}\,f^2(d_*(Y)m)\ge
1+c_4(d)\,\int_{2d_*(Y)}^\infty s^{d-1}f^2(s)\,ds.
\end{split}
\end{equation*}
The proof is complete.    \end{proof}

\begin{corollary}\label{cor1}
If $e_j\in\dom S_X(f)$ for some $j\in\N$ and all $X\in\kX_n$ then \eqref{maxdomain} holds.
\end{corollary}

\begin{remark} Condition \eqref{maxdomain} for an individual set $X$ has nothing to do with bound
\eqref{decay2}. Indeed, let $f$ tend to zero arbitrarily slowly as $x\to\infty$. Choose a sequence of
positive numbers $\{t_k\}$, $t_1=0$ so that $f(t_k)\le e^{-k}$. Now take a set $X=\{x_k\}_{k\in\N}$
with $x_k=t_k\xi$, $k\in\N$, $\xi$ a unit vector. Then
$$
 \sum_{i=1}^\infty f^2(|x_k - x_1|)\le \sum_k e^{-2k}<\infty
$$
regardless of whether condition \eqref{decay2} holds or not.
\end{remark}
The example below illustrates Proposition \ref{maxoperator} and Theorem \ref{bochbound}.

\begin{example}
Let $h(r)=(1+r)^{-1}\in CM_0(\R_+)$. Take $X=\Z_+^2=\{(p,q): p,q\in\Z_+\}$ labeled in
the following way
$$ X=\bigcup_{m=0}^\infty X_m, \qquad X_m=\{x_k^{(m)}\}_{k=0}^m, \quad x_k^{(m)}=(m-k,k), \quad X_0=\{(0,0)\}. $$
As $|x_k^{(m)}|^2=(m-k)^2+k^2=m^2+2k(k-m)\le m^2$, we can easily compute the sum in \eqref{maxdomain}
$$
\sum_{m=0}^\infty\sum_{k=0}^m h^2(|x_k^{(m)}|)=\sum_{m=0}^\infty\sum_{k=0}^m \frac1{(1+|x_k^{(m)}|)^2}\ge
\sum_{m=0}^\infty\sum_{k=0}^m \frac1{(1+m)^2}=\sum_{m=0}^\infty \frac1{1+m}=+\infty,
$$
which is consistent with Proposition \ref{maxoperator}, since $d=\dim\kL(X)=2$, and condition \eqref{decay2} is violated.

On the other hand, let $X=\Z_+$, so we come to a version of the well-known Hilbert--Toeplitz matrix
\begin{equation}\label{hilb-toep}
\kS_X(h)=\|(1+|i-j|)^{-1}\|_{i,j\in\N}, \qquad h(r)=\frac1{1+r}=\int_0^\infty e^{-sr}\,e^{-s}\,ds.
\end{equation}
Now $d=1$, so by Proposition \ref{maxoperator}, \eqref{maxdomain1} holds. Yet the operator $S_X(h)$ is unbounded in view of
Theorem \ref{bochbound} ($\Z_+$ is a $1$-regular set). We show later in Proposition \ref{toep2} that $S_X(h)$
is a positive definite and self-adjoint operator.
\end{example}

\medskip

An important property of a minimal Schoenberg operator $A=S_X(f)$
constitutes the content of the following theorem.
%%%%%%%%%%%%%%%%%%%%%%%%%%%%%%%%%%%%%
 \begin{theorem}\label{semibound_Schoenb}
Let $f\in \Phi_\infty(\alpha)$, $\alpha\in (0, 2]$, $X=\{x_j\}_{j\in\N}\subset \R^n$, and $X\in\mathcal X_n$. Assume
that the Schoenberg matrix ${\mathcal S}_{X}(f)$ satisfies condition \eqref{unboundher}. Then the $($minimal$)$
Schoenberg operator $S_{X}(f)$  associated with the matrix ${\mathcal S}_{X}(f)$ is a symmetric positive definite operator, i.e.,
\begin{equation}\label{strongposit}
\langle S_{X}(f)\xi,\xi\rangle\ge\varepsilon|\xi|^2, \qquad
\xi\in\dom S_X(f),\quad \varepsilon>0,
\end{equation}
and so the deficiency indices $n_{\pm}(A) =
\dim\ker\bigl(A^*\bigr)$. In particular, $S_X(f)$ is self-adjoint if
and only if $\ker A^*=\{0\}$.
\end{theorem}
%%%%%%%%%%%%%%%%%%%%%%%%%%%%%%%%%%%%%%%%%%%%%%%%%%%%
   \begin{proof}
According to  Theorem \ref{inftystrong},
the function $f \in \Phi_\infty(\alpha)$ is strongly $X$-positive
definite, i.e., there exists $\varepsilon>0$  such that
     \begin{equation}\label{2.26}
\sum_{j,k =1}^{N} f(|x_j - x_k|)\xi_j\overline{\xi}_k \ge
\varepsilon\,\sum_{j=1}^N | \xi_j|^2, \qquad \xi=\{\xi_j\}^N_1 \in
\Bbb C^N,\quad  \forall N\in\Bbb N.
\end{equation}
Due to assumption \eqref{unboundher} the basis $\{e_j\}_{j\in\N}$ is
a basis of the matrix representation of the minimal operator
$S_X(f)$ associated with the Schoenberg matrix ${\mathcal S}_X(f)$.
Therefore inequality  \eqref{2.26} means that for any finite vector
$\xi=(\xi_1,\xi_2, \ldots,\xi_N,0,0,\ldots)$
  \begin{equation*}
(A\xi,\xi) = (A'\xi,\xi)\ge \varepsilon\,\sum_{k=1}^N | \xi_k|^2 =
\varepsilon\,\|\xi\|^2.
    \end{equation*}
Taking the closure we get the statement.
   \end{proof}

Note that the proof of our main result about $\Phi_\infty$-functions-- Theorem \ref{inftystrong} --
in the next section is completely independent of the above Theorem \ref{semibound_Schoenb}.

The converse to Theorem \ref{semibound_Schoenb} is true in more general setting.

     \begin{proposition}\label{separ}
Assume that the Schoenberg matrix ${\mathcal S}_{X}(f)$, $f\in\Phi_n$, satisfies condition $\eqref{unboundher}$, and the
$($minimal$)$ Schoenberg operator $S_{X}(f)$  associated with the matrix ${\mathcal S}_{X}(f)$ satisfies $\eqref{strongposit}$,
i.e., it is positive definite. Then $X$ is separated, i.e., $d_*(X)>0$.
\end{proposition}
\begin{proof}
In the above assumptions one has
\begin{equation}\label{positive}
\langle S_X(f)h,h\rangle\ge c\|h\|^2, \qquad 0<c<\infty
\end{equation}
for each $h\dom(S_X(f))$. Hence putting $h=e_k-e_j\in\dom(S_X(f))$, we see that
$$ \langle S_X(f)h,h\rangle=2f(0)-2f(|x_k-x_j|)\ge 2c, $$
so $f(|x_k-x_j|)\le f(0)-c$, $c>0$, which immediately implies $d_*(X)>0$.
\end{proof}

Proposition \ref{separ} says that if $d_*(X)=0$ and $S_X(f)$ is bounded for $f\in\Phi_n$, then $0\in\sigma(S_X(f))$.
It is easy to manufacture such $X$ for $f(t)=e^{-t}$ (cf. \cite[Lemma 3.7]{MalSch12}).

There is a simple function theoretic analogue of Proposition \ref{separ}.

     \begin{proposition}\label{ftsepar}
If $f\in\Phi_n$ is strongly $X$-positive definite, then $X$ is separated.
\end{proposition}
\begin{proof}
By the definition we have for all $k,j$
$$ f(0)(|\xi_1|^2+|\xi_2|^2)-f(|x_j-x_k|)(\xi_1\bar\xi_2+\bar\xi_1\xi_2)\ge c(|\xi_1|^2+|\xi_2|^2). $$
By putting $\xi_1=\xi_2\not=0$ we see that
$$ f(0)-f(|x_j-x_k|)\ge c>0, $$
so $X\in\kX_n$, as needed.
\end{proof}

\subsection{Schoenberg--Toeplitz operators}

Although we have no sufficient conditions for general Schoenberg operators $S_X(f)$ to be self-adjoint,
Theorem \ref{semibound_Schoenb} gives an essential step toward proving self-adjointness, since it reduces
this problem to the study of $\ker A^*$.

  \begin{definition}
(i) Let $\N_0:= \N\cup\{0\}$.  Recall that a matrix $\mathcal A :=
\|a_{jk}\|_{j,k\in \N_0}$ is called a {\it Toeplitz matrix} if there is a
numerical sequence $\{a_m\}_{m\in \Z}$ such that $a_{jk}=a_{j-k}$
for every $j,k\in\N_0$.

(ii) An operator $A$ defined at least on the set of analytic
polynomials $Pol_+$ is called a Toeplitz operator if its matrix with
respect to the basis $\{z^k\}_{k\in\Bbb N_0}=\{e^{ik\varphi}\}_{k\in\Bbb
N_0}$ is the Toeplitz matrix.
    \end{definition}
It is known that a Toeplitz operator is characterized by the identity
          \begin{equation}
S^* A S= A,
          \end{equation}
where $S$ is a unilateral shift in $l^2$. According to the  basic
assumption \eqref{unboundher}  the Toeplitz matrix $\mathcal A$
defines an operator in $\ell^2$ if  $\{a_k\}\in l^2(\Bbb Z)$, i.e.,
       \begin{equation}\label{2.47}
\sum_{j\in\Bbb Z}|a_j|^2<\infty.
       \end{equation}
In this case the {\it Toeplitz symbol} is a function given by
   \begin{equation}\label{2.48}
a(A,e^{i\varphi}):=\sum_{k\in\Z} a_ke^{ik\varphi}\in L^2(\Bbb T),
\qquad \varphi\in [-\pi,\pi],
  \end{equation}
\begin{lemma}\label{Toeplitz_Lemma}
Let $a_{-j}=  \bar {a_{j}}$, $j\in \N$, i.e., the Toeplitz matrix $\mathcal A=\|a_{j-k}\|_{j,k\in\Bbb N_0}$ is a
Hermitian matrix. Assume also that $\mathcal A$ satisfies \eqref{2.47} and  the  minimal symmetric
Toeplitz operator $A$ associated in $l^2(\Bbb N)$ with $\mathcal A$ is semibounded from below.  Then it is self-adjoint, $A= A^*$.
   \end{lemma}
\begin{proof}
Without  loss of generality we can assume that  $A$ is positive
definite.  In this case  it suffices to make sure that the conjugate
(maximal) operator $A^*$ has the trivial kernel. Since $A^*=A_{max}$
acts by means of the same matrix $\kS_X(f)$ (but defined  on the
maximal domain), the latter property is equivalent to
\begin{equation}\label{trivkern}
\begin{bmatrix}
a_0& a_1& a_2&\ldots \\
a_{-1}& a_0&a_1&\ldots \\
a_{-2}& a_{-1}& a_0&\ldots \\
\vdots&\vdots&\vdots& &
\end{bmatrix}
\begin{bmatrix}
p_0 \\ p_1 \\ p_2 \\ \vdots
\end{bmatrix}=\mathbb{O} \Rightarrow p_j\equiv 0, \quad p=\{p_j\}\in\ell^2.
\end{equation}

To prove implication \eqref{trivkern} it is instructive to rephrase
the problem in the function theoretic terms.

Let $U$ denote  the multiplication (shift) operator on $L^2(\T)$.
The equality in \eqref{trivkern} means that the function
$$
p(t):=\sum_{j\ge0} p_jt^j\in H^2
$$
is orthogonal to the system $\{U^k a\}_{k\ge0}$, where $a\in L^2(\T)$
is the Toeplitz symbol \eqref{2.48}.
 In other words, $p\,a\in L^1(\T)$ is orthogonal to all powers
$\{t^k\}_{k\ge0}$, i.e., $p_-:= p\,a\in H^1_-$.

A positive definiteness of the minimal operator $A$ reads as follows
\begin{equation}\label{semibound}
(Aq,q)=\sum_{k,j=0}^N a_{k-j} q_j\bar q_k=\int_{\T}
a(t)|q(t)|^2\,m(dt)\ge \varepsilon \|q\|^2_{L^2(\T)}, \qquad
q(t):=\sum_{j=0}^N q_jt^j, \quad \varepsilon>0,
\end{equation}
for an arbitrary analytic polynomial $q$, $m$ is the normalized Lebesgue measure on $\T$.
It is clear from \eqref{semibound} that $a(t) \ge \varepsilon$  for a.e.  $t = e^{i\varphi}\in \T$.
Therefore (see \cite[Theorem II.4.6]{Garn}) there is an outer function $D$ such that
$$
a(t) =|D(t)|^2, \qquad D\in H^2, \quad D^{-1}\in H^\infty.
$$
We have $p(t)\,a(t)=p(t)\,|D(t)|^2 = p_-(t)\in H^1_-$ and hence
$$
p(t)\,D(t) = \frac{p_-(t)}{\overline {D(t)}}\,.
$$
But the left-hand side of the latter equality belongs to $H^1$,
whereas the right-hand side lies in $H^1_-$ which yields $p\equiv
0$, as claimed. The proof is complete.
\end{proof}

A sequence $X=\{x_k\}_{k\in\N}\subset\R^n$ is called a {\it Toeplitz
sequence}, if $|x_i-x_j|=|i-j|$ for $i,j\in\N$. The latter is
equivalent (recall that by our convention $x_1=0$) to
$x_k=(k-1)\xi$, $\xi\in\R^n$, and $|\xi|=1$, so $d=\dim\kL(X)=1$. In
this case $S_X(f)$ is a Toeplitz operator, which will be called a {\it Schoenberg--Toeplitz operator}.
The Toeplitz symbol $a$ \eqref{2.48} takes now the form
  \begin{equation}\label{2.51}
a(f,e^{i\varphi}):=\sum_{k\in \Z} f(|k|)e^{ik\varphi}.
   \end{equation}

%%%%%%%%%%%%%%%%%%%%%%%%%%%%%%%%%%%%%%%%%%%%%%%%%%%%%%%%%

\begin{remark}
(i) Self-adjointness of not necessarily positive Toeplitz operators
with the Toeplitz symbol from $BMO(\T)$ was established by V. Peller
\cite{Pel89}. This is the case for the Hilbert--Toeplitz operator
\eqref{hilb-toep} with the Toeplitz symbol
$$ a(h,t)=1-2\Re\,\frac{\log(1-t)}{t}\in BMO(\T), $$
but not for general Schoenberg--Toeplitz operators with Toeplitz
symbols \eqref{symbol}.

(ii) Semi-bounded Toeplitz  operators have been studied in
several papers (see  \cite{RoRo} and references therein). For
instance, it is proved in \cite{Ros60} that
the Friedrichs extension $A_F$ of $A$ has absolutely continuous
spectrum. However, according to Lemma \ref{Toeplitz_Lemma}, $A_F=A$.
\end{remark}

In the rest of the section we will focus on the Schoenberg--Toeplitz operators $S_X(f)$ with
symbols $f\in \Phi_\infty = \Phi_\infty(2)$. We clarify and complete Corollary \ref{alphaclass} for
such operators and describe their spectra in terms of the Schoenberg measures $\sigma=\sigma_f$.

%%%%%%%%%%%%%%%%%%%%%%%%%%%%%%%%%
      \begin{proposition}\label{toep1}
Let $f\in\Phi_\infty$ and let  $\sigma$ be its Schoenberg measure
\eqref{alphacm}. The Schoenberg--Toeplitz matrix $\mathcal
S_X(f)$ defines a minimal operator $S_X(f)$ in $\ell^2$ if and only
if $f\in L^2(\R_+)$. In this case $S_X(f)$ is self-adjoint, its
spectrum is purely absolutely continuous and fills in the interval
\begin{eqnarray}\label{spectrschtoepl}
\sigma(S_X(f)) = \sigma_{ac}(S_X(f))=[c_-, c_+], \qquad \qquad \qquad \nonumber  \\
0<c_-:=\int_0^\infty \vartheta_3\bigl(\pi,
e^{-s}\bigr)\,\sigma(ds)<c_+:=\int_0^\infty \vartheta_3\bigl(0,
e^{-s}\bigr)\,\sigma(ds) \le +\infty,
  \end{eqnarray}
where $\vartheta_3$ is the Jacobi theta-function.

Moreover, the operator $S_X(f)$ is bounded if and only if $f\in
L^1(\R_+)$, or, equivalently,
  \begin{equation}\label{negatmom1}
\int_0^\infty \frac{\sigma(ds)}{\sqrt{s}}<\infty.
  \end{equation}
  \end{proposition}
%%%%%%%%%%%%%%%%%%%%%%%%%%%%%%%%%%%%
   \begin{proof}
As the Schoenberg symbol $f$ is a nonnegative and monotone
decreasing function, conditions $f\in L^2(\R_+)$  and $\{f(k)\}_{k\ge0}\in\ell^2$ are equivalent, so \eqref{2.47} is met.
Next, for $f\in\Phi_\infty$ the corresponding minimal operator is symmetric and strongly positive definite
by Theorem \ref{semibound_Schoenb}. Hence $S_X(f)$ is self-adjoint in view of Lemma \ref{Toeplitz_Lemma}.

Consider the kernel function $e_s(u):=e^{-su^2}$, $s>0$, so
$S_X(e_s) = \|e^{-s|i-j|^2}\|_{i,j\in\N}$. Since $e_s(\cdot)\in L^1(\R_+)$, the operator $S_X(e_s)$ is bounded
by Theorem \ref{bochbound}.  The corresponding Toeplitz symbol is given by \eqref{2.51}.
It can now be expressed by means of the Jacobi theta-function
   \begin{equation*}
a(e_s,e^{i\varphi}) =\sum_{k\in\,\Z}
e^{-s|k|^2}e^{ik\varphi}=\vartheta_3\Bigl(\frac{\varphi}2,
e^{-s}\Bigr).
   \end{equation*}
It is well known (see \cite[Chapter 21]{WhWa}) that $\vartheta_3$ is
positive on the real line and
$$ \frac{\vartheta_3'(\varphi)}{\vartheta_3(\varphi)}=-4\sin 2\varphi\,\sum_{k=1}^\infty \frac{q^{2k-1}}{1+2q^{2k-1}\cos 2\varphi+q^{4k-2}}\,,
\qquad q=e^{-s}, $$ so $a(e_s)$ is monotone decreasing on $[0,\pi]$
($a(e_s)$ is ``bell-shaped'' on $[-\pi,\pi]$).
By the Hartman--Wintner theorem (see, e.g., \cite[Theorem 4.2.7]{Nik1}) its spectrum agrees with
the range of $a(f)$, so it is the interval
$$ \sigma(S_X(e_s))=a(e_s,\T)=[a(e_s,-1), a(e_s,1)]=[\vartheta_3\Bigl(\frac{\pi}2,q\Bigr), \vartheta_3(0,q)]. $$

For a general function $f\in\Phi_\infty$ the Toeplitz symbol $a(f)$
of $S_X(f)=\|f(|i-j|)\|_{i,j\in\N}$ can be computed as
\begin{equation}\label{symbol}
a(f,e^{i\varphi})=\sum_{k\in\,\Z}
f(|k|)e^{ik\varphi}=\sum_{k\in\,\Z} e^{ik\varphi}\int_0^\infty
e^{-s|k|^2}\,\sigma(ds)= \int_0^\infty
\vartheta_3\Bigl(\frac{\varphi}2, e^{-s}\Bigr)\,\sigma(ds), \qquad
\varphi\not=0. \end{equation}
It is easily seen that
   \begin{equation}\label{symbol-Inequality}
 \vartheta_3\Bigl(\frac{\pi}2,e^{-s}\Bigr)\le
\vartheta_3\Bigl(\frac{\varphi}2,e^{-s}\Bigr)\le
\vartheta_3(0,e^{-s}) =\sum_{k\in\Z} e^{-sk^2}\sim \frac1{\sqrt{s}},
\quad s\to +0, \qquad 0\le\varphi\le\pi.
  \end{equation}
Again by the Hartman--Wintner theorem, the
spectrum of $S_X(f)$ agrees with the range of $a(f)$, which is exactly
the interval given  by \eqref{spectrschtoepl}. Its absolute continuity is a standard fact
in the theory of Toeplitz operators, (see, e.g., \cite[p. 64]{RoRo}).

By Theorem \ref{bochbound} the boundedness of $S_X(f)$ is equivalent to  $f\in
L^1(\R_+)$. In turn, the latter is equivalent to \eqref{negatmom1} by Corollary
\ref{alphaclass}, applied with $\alpha =2$ and $d=1$. The proof is complete.
   \end{proof}

It is easy to express the inclusion $f\in \Phi_\infty\cap L^2(\R_+)$ in terms of
$\sigma$ (cf. \eqref{negatmom1})
   \begin{equation}\label{unboundtoepl1}
\int_{\R_+^2}
\frac{\sigma(ds_1)\,\sigma(ds_2)}{\sqrt{s_1+s_2}}<\infty.
   \end{equation}

%%%%%%%%%%%%%%%%%%%%%%%%%%%%%%%%%%%%%%

Next, we provide a similar result for $f\in CM_0(\R_+)$.
  \begin{proposition}\label{toep2}
Let $f\in CM_0(\R_+)$, $\tau$ be its Bernstein measure
\eqref{bernstein}.  The Schoenberg--Toeplitz matrix $\mathcal
S_X(f)$ defines a minimal operator $S_X(f)$ in $\ell^2$ if and only
if $f\in L^2(\R_+)$. In this case $S_X(f)$ is self-adjoint, its
spectrum is purely absolutely continuous and fills in the interval
   \begin{equation}%\label{spectrschtoepl}
\sigma(S_X(f)) = \sigma_{ac}(S_X(f)) = [c_-, c_+], \qquad
0<c_{\pm}=\int_0^\infty \frac{1\pm e^{-s}}{1\mp e^{-s}}\,\tau(ds).
  \end{equation}
Moreover, the  operator $S_X(f)$ is bounded if and only if $f\in
L^1(\R_+)$, or, equivalently,
  \begin{equation}\label{negatmom}
\int_0^\infty \frac{\tau(ds)}{s}<\infty.
   \end{equation}
\end{proposition}
%%%%%%%%%%%%%%%%%%%%%%%%%%%%%%%%%%%%%%%%%%%%%%%%%%%
  \begin{proof}
As in the proof of the preceding result, we start  with the kernel
function $e_s(u):=e^{-su}$, $s>0$ and relate the Schoenberg and
Toeplitz symbols:
   \begin{equation*}
a(e_s,e^{i\varphi}) =\sum_{k\in\,\Z}
e^{-s|k|}e^{ik\varphi}=1+\frac{e^{-s+i\varphi}}{1-e^{-s+i\varphi}}+\frac{e^{-s-i\varphi}}{1-e^{-s-i\varphi}}
= \frac{1-e^{-2s}}{|1-te^{-s+i\varphi}|^2}= P(e^{-s},e^{i\varphi}),
  \end{equation*}
where $P(e^{-s},e^{i\varphi})$ denotes  the Poisson kernel for the unit disk.
Hence $S_X(e_s) = \|e^{-s|i-j|}\|_{i,j\in\N}$ is bounded and its
spectrum is the interval
$$
\sigma(S_X(e_s))=a(e_s,\T)=\left[\frac{1-e^{-s}}{1+e^{-s}}, \, \frac{1+e^{-s}}{1-e^{-s}}\right].
$$
The Toeplitz symbol $a(f)$ of the operator
$S_X(f)=\|f(|i-j|)\|_{i,j\in\N}$ can be computed as above
\begin{equation}\label{symbol}
a(f,e^{i\varphi})=\sum_{k\in\,\Z} f(|k|)e^{ik\varphi}=\sum_{k\in\,\Z} e^{ik\varphi}\int_0^\infty
e^{-s|k|}\,\tau(ds)= \int_0^\infty P(e^{-s},e^{i\varphi})\,\tau(ds), \qquad
\varphi\not=0.
\end{equation}

One completes the proof in just the same fashion as in Proposition \ref{toep1}.
     \end{proof}

Similarly, the condition $f\in CM_0(\R_+)\cap L^2(\R_+)$ is equivalent to (cf. \eqref{negatmom})
\begin{equation}\label{unboundtoepl}
\int_{\R_+^2} \frac{\tau(ds_1)\,\tau(ds_2)}{s_1+s_2}<\infty.
\end{equation}

%%%%%%%%%%%%%%%%%%%%%%%%%%%%%%%%%%%%%%%%%%%%%%%
       \begin{example}\label{nonell2}
It is not hard to manufacture a Schoenberg--Toeplitz matrices with the Schoenberg symbol
$f\in CM_0(\R_+)\backslash L^2(\R_+)$. Indeed, one can take
\begin{equation}%%%%\label{hilb-toepNew}
 \kS_X(f_\beta) =\|(1+|i-j|)^{-\beta}\|_{i,j\in\N}, \qquad
f_\beta(r)=\frac1{(1+r)^\beta}= \frac1{\Gamma(\beta)}\int_0^\infty
e^{-sr}\,s^{\beta-1}e^{-s}\,ds
\end{equation}
with $0<\beta\le1/2$. In this example no coordinate vector $e_j$,
$j\in \N$, belongs to $\ell^2$.
\end{example}
   \begin{remark}
(i). According to a result of Brown and Halmos (see, e.g., \cite[Theorem
4.1.4]{Nik1}) the operator $S_X(f)$ is bounded if and
only if $a(f)\in L^\infty(\T)$. Due to the asymptotic relation
\eqref{symbol-Inequality} for $f\in \Phi_\infty$ the latter is
equivalent to \eqref{negatmom1}. This observation provides another
proof of the last statement of both preceding propositions.

(ii). The relation between the Schoenberg symbol $f\in\Phi_{\infty}(\alpha)$ for $\alpha=1,2$
and the Toeplitz symbol $a(f)$ is implemented by the Poisson kernel and the Jacobi theta-function, respectively.
We are unaware of the similar relation for $1<\alpha<2$.

(iii). A Schoenberg--Toeplitz operator $S_X(f)$ with $f\in\kM_+$ is bounded if and only if
the Fourier coefficients of its Toeplitz symbol $a(f)$ \eqref{2.51} are positive and monotone
decreasing and $a(f)\in W$, the Wiener algebra of absolutely convergent Fourier series.
This result stems directly from Theorem \ref{bochbound}.
   \end{remark}

\begin{example}\label{uninvertible}
We construct a bounded Schoenberg--Toeplitz operator $S_X(\varphi)$ with $0\in\sigma(S_X(\varphi))$.
Take any Toeplitz sequence $X\subset\R^1$ so that $|x_i-x_j|=|i-j|$  and put
\begin{equation*}
\varphi(t)=\biggl(1-\frac{t}2\biggr)_+\in\Phi_1, \qquad
(a)_+:=\max(a,0).
\end{equation*}
Then $S_X(\varphi)=J(\{1/2\}, \{1\})$ is the Jacobi operator with $1$ on the main diagonal and $1/2$ off the main diagonal.
It is well known that $\sigma(S_X(\varphi))=[0,2]$, as claimed. Certainly, $\varphi\notin\Phi_\infty$.
\end{example}

\section{Schoenberg matrices and harmonic analysis on $\R^n$}\label{classes}

\subsection{Radial strongly  $X$-positive definite  functions}\label{rpdfstrong}

We begin with some basics of harmonic analysis on the Hilbert spaces (\cite[Section C.3.3]{Nik2}, \cite{young80}).

    \begin{definition}\label{def4.1}
Let $\kF=\{f_k\}_{k\in\N}$ be a sequence of vectors in a Hilbert
space $\kH$.
\begin{itemize}
\item[(i)] $\kF$ is called a {\it Riesz--Fischer sequence}
if for all $(\xi_1,\cdots,\xi_m)\in \C^m$ and $m\in\N$ there is a constant $c>0$ such that
   \begin{align}\label{rieszfischerse}
\bigg\|\sum_{k=1}^m \xi_k f_k\bigg\|^2_\kH \geq  c~ \sum_{k=1}^m |\xi_k|^2.
  \end{align}

\item[(ii)] $\kF$ is called a {\it Bessel sequence}
if for all $(\xi_1,\cdots,\xi_m)\in \C^m$ and $m\in\N$ there is a constant $C<\infty$ such that
   \begin{align}\label{besselse}
\bigg\|\sum_{k=1}^m \xi_k f_k\bigg\|^2_\kH  \leq C~ \sum_{k=1}^m|\xi_k|^2.
   \end{align}

\item[(iii)] $\kF$ is called a {\it Riesz sequence} (or a Riesz basis in its linear span) if $\kF$
is both Riesz--Fischer and Bessel sequence. If $\kF$ is complete we say about a Riesz basis in $\kH$.

\if{\item[(iv)]  $\kF$ is called a {\it Bari sequence}  (or a Bari basis in its linear span)
if there exists an orthonormal sequence $\{e_j\}_{j\in\N}$ of $\kH$ such that
    \begin{equation}
\sum_{j\in \N}\|f_j-e_j\|^2<\infty.
    \end{equation}
}\fi
\end{itemize}
    \end{definition}

It turns out that the above notions applied to sequences of exponential functions in $L^2$-spaces are tightly
related to the strong $X$-positive definiteness.

Given an arbitrary sequence $X=\{x_k\}_{k\in\N}$ of distinct points in $\R^n$, we introduce a system
\begin{equation}\label{expfun}
 \kE_X =\{e(\cdot,x_k)\}_{k\in\N}, \qquad e(x,x_k)=e^{i(x,x_k)}, \quad x\in\R^n, \end{equation}
of exponential functions.

\begin{proposition}\label{prop_Riesz_seq_in_space}
Let $g$ be a positive definite function $(\ref{bochequa})$ with the Bochner measure $\mu$. For an arbitrary sequence
$X=\{x_k\}_{k\in\N}$ of distinct points in $\R^n$ and for the system of exponential functions $\kE_X$ $\eqref{expfun}$ the
following holds.

\begin{itemize}
\item[\em (i)]  $\kE_X$ is a Riesz--Fischer sequence in $L^2(\R^n,\mu)$ if and only if  $g$ is strongly $X$-positive definite.

\item[\em (ii)] $\kE_X$ is a Bessel sequence if and only if the Gram matrix
\begin{equation}\label{grammatrix}
Gr (\kE_X,L^2(\R^n,\mu)) = \|\langle e(\cdot,x_k),
e(\cdot,x_j)\rangle_{L^2(\R^n,\mu)}\|_{k,j\in \N} =
\|g(x_k-x_j)\|_{k,j\in \N}
 \end{equation}
defines a bounded, self-adjoint and nonnegative operator on $\ell^2$.

\item[\em (iii)] $\kE_X$ is a Riesz sequence if and only if $Gr (\kE_X,L^2(\R^n,\mu))$ defines a bounded
and invertible, nonnegative operator.
   \end{itemize}
      \end{proposition}
%%%%%%%%%%%%%%%%%%%%%%%%%%%%%%%%%%%%
   \begin{proof}
It is clear that
\begin{equation}\label{equrfbs}
\sum_{k,j=1}^{m}g(x_k-x_j)\xi_j\overline{\xi}_k=
\int_{\R^n}\left|\sum_{k=1}^{m}\xi_k e(u,x_k)\right|^2\mu(du)
 = \bigg\|\sum_{k=1}^{m}\xi_k e(\cdot,x_k)\bigg\|_{L^2(\R^n,\mu)}^2
   \end{equation}
for $\xi=\{\xi_1,\dots,\xi_m\}\in \C^m$ and arbitrary $m\in \N$.
All statements are immediate from  \eqref{equrfbs}.
     \end{proof}
%%%%%%%%%%%%%%%%%%%%%%%%%%%%%%%%%%%%%%%%%%%%%%%%%%%%%%%%%%%%

The same system $\kE$ can be viewed as a system of vectors in another Hilbert space, namely $L^2(S_r^{n-1})$,
$S_r^{n-1}$ is a sphere in $\R^n$ of radius $r$, centered at the origin, with the normalized Lebesgue measure.
We denote this system by $\kE_X(S_r^{n-1})$. Such approach leads to RPDF's (see \cite{gmo2}).

The following result is borrowed  from \cite[Proposition 2.14]{MalSch12}.
We present it with the proof because of its importance in the sequel.

%%%%%%%%%%%%%%%%%%%%%%%%%%%%%%%%%%%%%%%%%%%%%%%%%%%%%%%%%%%%%%%%%%%%%%%%%
\begin{proposition}\label{rieszfisher}
Let $f\in\Phi_n$, $n\ge2$, with the measure $\nu=\nu(f)$ in $\eqref{schoenberg1}$.
Given an arbitrary sequence $X=\{x_k\}_{k\in\N}$ of distinct points in $\R^n$, the function $f$
is strongly $X$-positive definite if and only if there exists a Borel set  $\mathcal K \subset (0,+\infty)$, $\nu(\mathcal K)>0$
such that the system $\kE_X(S_r^{n-1})$ forms a Riesz--Fischer sequence for each $r\in \mathcal K$.
In particular, the function $f_\rho(\cdot)=\Omega_n(\rho\cdot)$, $\rho>0$, is strongly $X$-positive definite if and only if
the system $\kE_X(S_\rho^{n-1})$ is a Riesz--Fischer sequence.
   \end{proposition}
\begin{proof} It follows from \eqref{schoenberg1} and \eqref{fouriersphere} that for
$(\xi_1,\dots,\xi_m)\in \C^m$ and  $m\in \N$
  \begin{equation}\label{intfxkxjAA}
\sum_{j,k =1}^{m}f(|x_k - x_j|)\xi_j\overline{\xi}_k=
 \int_{0}^{+\infty}\left(\;
 \int_{S^{n-1}_r}\left|\sum_{k=1}^{m}\xi_k e(u,x_k)\right|^2\sigma_n(du)
 \right)\nu(dr).
 \end{equation}

Suppose that there exists a set $\mathcal K$ as stated above.
Then for every $r \in \mathcal K$ there is a constant $c(r)>0$ so that
  \begin{equation}\label{2.18B}
\int_{S^{n-1}_r}\left|\sum_{k=1}^{m}\xi_k
e(u,x_k)\right|^2\sigma_n(du)= \bigg\|\sum_{k=1}^{m}\xi_k
e(\cdot,x_k)\bigg\|^2_{L^2(S_r^{n-1})} \ge  c(r) \sum_{k=1}^m
|\xi_k|^2.
     \end{equation}
Choosing $c(r)$ bounded and measurable and
combining the latter inequality with  \eqref{intfxkxjAA}, we obtain

  \begin{equation}\label{intfxkxjAB}
\begin{split}
\sum_{j,k =1}^{m}f(|x_j - x_k|)\xi_j\overline{\xi}_k &\ge
 \int_{\mathcal K} \left(\bigg\|\sum_{k=1}^{m}\xi_k e(\cdot,x_k)\bigg\|^2_{L_r^2(S^{n-1})}
 \right)\nu(dr)  \ge  c
\sum_{k=1}^m | \xi_k|^2, \\  c &:= \int_{\mathcal K}c(r)\nu(dr). \end{split}
 \end{equation}
Since $\nu(\mathcal K)>0$ and $c(r)>0$, we  have $c>0$, so $f$ is strongly $X$-positive definite.

Conversely, if
$$ \int_0^\infty h(r)\,\nu(dr)\ge c_1>0, \qquad h(r)=\bigg\|\sum_{k=1}^{m}\xi_k e(\cdot,x_k)\bigg\|^2_{L_r^2(S^{n-1})}, $$
then there is  a Borel set  $\mathcal K \subset (0,+\infty)$ of positive $\nu$-measure such that $h\ge c_1$ on $\mathcal K$,
as claimed.
\end{proof}

We want to lay stress on the fact that the measure $\nu$ enters this result only via existence of a certain Borel set $\kK$
of positive $\nu$-measure.

\begin{corollary}\label{mutac}
Let $f_j\in\Phi_n$, $n\ge2$, $j=1,2$, with the measures $\nu_1$ and $\nu_2$ in $\eqref{schoenberg1}$,
respectively. Assume that $\nu_1$ is absolutely continuous with respect to $\nu_2$. Given a set
$X=\{x_k\}_{k\in\N}$ of distinct points in $\R^n$, if $f_1$ is strongly $X$-positive definite then
so is $f_2$. In particular, if $\nu_1$ and $\nu_2$ are mutually  absolutely continuous $($equivalent$)$, then $f_1$ and $f_2$
are strongly $X$-positive definite simultaneously.
\end{corollary}

\begin{proof}
By Proposition \ref{rieszfisher} there is a Borel set $\mathcal K \subset (0,+\infty)$,
$\nu_1(\mathcal K)>0$ so that the system $\kE_r=\{e(\cdot,\,rx_k)\}_{k\in\N}$ forms a Riesz--Fischer sequence
in $L^2(S^{n-1})$ for each $r\in \mathcal K$. Since $\nu_1$ is absolutely continuous with respect to $\nu_2$, then $\nu_2(\kK)>0$ as well.
Now Proposition \ref{rieszfisher} applies in backward direction and yields strong $X$-positive definiteness of $f_2$, as claimed.
\end{proof}

\medskip

%%%%%%%%%%%%%%%%%%%%%%%%%%%%%%%%%%%

We are in a position now to prove the main result of the section.

\begin{theorem}[=Theorem \ref{Intro_stronglydef}]\label{propositionstronglydef}
Let $(\const \not = )f\in\Phi_n$, $n\ge2$, with the representing measure $\nu = \nu(f)$
from $\eqref{schoenberg1}$. If $\nu$ is equivalent to the Lebesgue measure on $\R_+$,  then $f$ is strongly
$X$-positive definite for each $X\in\kX_n$.
   \end{theorem}

\begin{proof}
We begin with a function $f_s(r):=e^{-sr}\in\Phi_n$ and show that for each $X\in\kX_n$ $f_s$ is strongly
$X$-positive definite for all large enough $s>0$. Indeed, take $s$ so that
$$ \|t^{n-1}f_s\|_{L^1(\R_+)}=\int_0^\infty t^{n-1}e^{-st}\,dt=\frac{\Gamma(n)}{s^n}<\frac{d_*^n(X)}{5^nn^2}. $$
By Theorem \ref{bochbound} (see \eqref{boundforinverse}) the Schoenberg operator $S_X(f_s)$ is bounded and invertible,
so \eqref{stronglyxpositive} holds, as needed.

To make use of Corollary \ref{mutac} we compute the measure $\nu(f_s)$.  To this end recall a well-known
result from the Fourier transforms theory, which plays a key role in the sequel.

Let $h\in L^1(\R^n)$ and let $\widehat h$ be its Fourier transform
$$ \widehat h(t):=\frac1{(2\pi)^{n/2}}\,\int_{\R^n} h(x)e^{-i(t,x)}\,dx. $$
If $h(\cdot)= h_0(|\cdot|)$ is a radial function, then so is
$\widehat h(\cdot)=H_0(|\cdot|)$. Moreover, $H_0$ and $h_0$ are
related by (see, e.g., \cite[Theorem IV.3.3]{Stein-Weiss})
\begin{equation}\label{radialft}
H_0(r)=\frac1{r^q}\int_0^\infty
J_q(ru)u^{q+1}h_0(u)\,du=\frac1{2^q\Gamma(q+1)}\int_0^\infty
\Omega_n(ru)u^{n-1}h_0(u)\,du, \quad q:=\frac{n}2-1.
\end{equation}
The latter is usually referred to as the Fourier--Bessel transform.

We apply \eqref{radialft} to a pair of functions
$$ h(x)=\frac{2^{n/2}\Gamma\bigl(\frac{n+1}2\bigr)}{\sqrt{\pi}}\,\frac{s}{(s^2+|x|^2)^{\frac{n+1}2}}\,,
\qquad \widehat h(t)=e^{-s|t|}, $$
(this is a particular case of \eqref{furtrans} below) and come to
\begin{equation}\label{schoenberg4}
f_s(r) =
e^{-sr}=\frac2{B\bigl(\frac{n}2,\frac12\bigr)}\,
\int_0^\infty\Omega_n(ru)\frac{su^{n-1}}{(s^2+u^2)^{\frac{n+1}2}}\,du,
\quad s,t>0, \quad B(a,b):=\frac{\Gamma(a)\Gamma(b)}{\Gamma(a+b)}
\end{equation}
is the Euler beta-function. This is exactly representation \eqref{schoenberg1} of $f_s$  with the measure
$$ \nu = \nu(f_s) = \frac{2}{B\bigl(\frac{n}2,\frac12\bigr)}\,
\frac{su^{n-1}}{{(s^2+u^2)^{\frac{n+1}2}}}\,du, $$
equivalent to the Lebesgue measure.
By the assumption of the theorem  the measures $\nu(f)$ and $\nu(f_s)$ are equivalent. Since
$f_s$ is strongly $X$-positive definite for large enough $s$ and each separated set $X\in\kX_n$, then  by
Corollary \ref{mutac}, so is $f$, as claimed.
\end{proof}
%%%%%%%%%%%%%%%%%%%%%%%%%%%%%%%%%%%%%%%%%%%%%%%%%%%%%%
  \begin{remark}
In fact, Theorem  \ref{propositionstronglydef} remains valid whenever the Lebesgue measure on $\R_+$
is absolutely continuous with respect to the measure $\nu$, that~is,
\begin{equation}\label{abscon}
\nu(ds)=\nu_{ac}+\nu_{sing}=\nu'(s)\,ds+\nu_{sing}, \qquad \nu'(s)>0 \ \ {\rm a.e.},
\end{equation}
$\nu_{sing}$ is a singular measure. This statement is immediate from the
obvious identity $S_X(f) = S_X(f_{ac}) + S_X(f_{sing})$, where $f_{ac}$  and $f_{sing}$ are the $\Phi_n$-functions defined by
\eqref{schoenberg1} with the measures $\nu_{ac}$ and $\nu_{sing}$, respectively. It is also a
consequence of Corollary \ref{mutac}, applied in its full extent.
   \end{remark}

%%%%%%%%%%%%%%%%%%%%%%%%%%%%%%%%%%%%%%%%%%%%%%%
\begin{theorem}[=Theorem \ref{intrinftystrong}]\label{inftystrong}
Let $f\in \Phi_\infty(\alpha)$, $0<\alpha\le2$, and $X\in\kX_n$. Then

\begin{itemize}
  \item [\em (i)]
  $f$  is strongly $X$-positive definite. In particular, if $\kS_X(f)$  generates an
  operator $S_X(f)$ on $\ell^2$, then it is positive definite and so invertible.

  \item [\em (ii)]
If the Schoenberg measure $\sigma=\sigma_f$ in $\eqref{alphacm}$ satisfies
\begin{equation}\label{alphabound1}
\int_0^\infty s^{-\frac{d}{\alpha}}\,\sigma(ds)<\infty, \qquad
d=\dim\kL(X),
\end{equation}
then the Schoenberg matrix $\kS_X(f)$  generates a bounded
$($necessarily invertible$)$ operator.

  \item [\em (iii)]
Conversely, let $S_Y(f)$ be bounded for at least one
$\delta$-regular set $Y$. Then \eqref{alphabound1} holds.
\end{itemize}
  \end{theorem}

\begin{proof}
(i). We apply again \eqref{radialft}, now to the pair of functions
$$ h(x)=(2s)^{-n/2}\exp\biggl(-\frac{|x|^2}{4s}\biggr), \qquad \widehat h(t)=e^{-s|t|^2}, $$
to obtain representation \eqref{schoenberg1} for $g_s$
\begin{equation}\label{schoenberg3}
g_s(r) := e^{-sr^2}=\frac{1}{2^q\Gamma(q+1)}\,
\int_0^\infty\Omega_n(ru)\frac{u^{n-1}}{(2s)^{n/2}}\,\exp\biggl(-\frac{u^2}{4s}\biggr)\,du,
\qquad r,s>0,
\end{equation}
(cf. \cite[Section V.4.3]{Akh65}).  Hence for any $g\in\Phi_\infty$ we can relate integral
representations \eqref{schoenberg1} and \eqref{alphacm}. Namely, combining \eqref{schoenberg3} with \eqref{alphacm}
we arrive at representation \eqref{schoenberg1} for $g\in\Phi_\infty$
\begin{equation}\label{shoenbinfty}
%\begin{split}
g(r) = \int_0^\infty\Omega_n(ru)\phi_{n,\sigma}(u)\,du, \quad
\phi_{n,\sigma}(u) = \frac{u^{n-1}}{2^q\Gamma(q+1)}\int_0^\infty
(2s)^{-n/2}\exp\biggl(-\frac{u^2}{4s}\biggr)\,\sigma(ds).
%\end{split}
\end{equation}
Clearly, $\nu(g)$ is equivalent to the Lebesgue measure, and the density $\phi_{n,\sigma}$ is bounded, strictly positive and
continuous on $\R_+$. The rest is Theorem \ref{propositionstronglydef}.

(ii). By Corollary \ref{alphaclass}, the Schoenberg operator $S_X(f)$ is bounded. It is invertible in view of the strong $X$-positive
definiteness of $f$.

(iii) is a combination of Theorem \ref{bochbound}, (iii), and  Corollary \ref{alphaclass}. The proof is complete.
\end{proof}
%%%%%%%%%%%%%%%%%%%%%%%%%%%%%%%%%%%%%%%%%%%%%%%%%%%%%%%%%%%
  \begin{remark}
As a special case of Theorem \ref{inftystrong}  we get that the function $g_{s}$ (see \eqref{schoenberg3}) is strongly
$X$-positive definite for all $s>0$ and each $X\in\kX_n$. The corresponding Schoenberg operator $S_X(g_{s})$ is bounded
and invertible by Theorem \ref{inftystrong}.
   \end{remark}
%%%%%%%%%%%%%%%%%%%%%%%%%%%%%%%%%%%%%%%%%%%%%%%
%%%%%%%%%%%%%%%%%%%%%%%%%%%%%%%%%%%%%%%%%%%%%%%
  \begin{example}
According to representation \eqref{alphacm}  each  $f\in \Phi_\infty(\alpha)$ is monotone decreasing. The following
example demonstrates that the monotonicity  is not necessary for $f$ to be  strongly $X$-positive definite for each separated set
$X\in\kX_n$. In particular, it gives an example of strongly $X$-positive definite function from $\Phi_n\backslash\Phi_\infty$.

Let  $K_\mu$ be the modified Bessel function of the second kind and
order $\mu$ (the definition and properties of $K_\mu$ are given in the next section).
By \cite[p.435, (5)]{Wat} the following  integral representation holds for $n\ge3$

  \begin{equation}\label{3.21}
h_s(r) :=
\Omega_n(rs)M_q(rs)=\frac{2(2s)^{n-2}}{B\bigl(q,\frac12\bigr)}\,\int_0^\infty\Omega_n(ru)\,\frac{u^{n-1}}{(u^4+4s^4)^{\frac{n-1}2}}\,du,
\quad M_q(t):=\frac{t^qK_q(t)}{2^{q-1}\Gamma(q)}
   \end{equation}
is the Whittle--Mat\'ern function, well-established in spatial statistics, $q=n/2-1$, $s>0$ is a parameter. We show later that
$M_q\in\Phi_\infty$, so the function $h_s\in \Phi_n$. Its representing measure $\nu(h_s)$ in \eqref{schoenberg1}
is equivalent to the Lebesgue measure and given explicitly by
$$ \nu(h_s)=\frac{2(2s)^{n-2}}{B\bigl(q,\frac12\bigr)}\,\frac{u^{n-1}}{(u^4+4s^4)^{\frac{n-1}2}}\,du $$
so by Theorem \ref{propositionstronglydef} $h_s$ is strongly $X$-positive definite function for each $X\in\kX_n$.

On the other hand, $h_s$ has infinitely many real zeros, so it is not monotone decreasing and hence  $f\not \in \Phi_\infty$. Thus, by
\eqref{3.21}, $f\in \Phi_n\backslash\Phi_\infty$.
\end{example}
\begin{remark}
If a real-valued function $f$ obeys $|f(r)|\le ce^{-ar}$, $a>0$, (as in the above example),
then by Proposition \ref{nonmonoton}, the Schoenberg operator $S_X(f)$ is bounded for each $X\in\kX_n$ \emph{and any} $n\in\N$.
   \end{remark}
%%%%%%%%%%%%%%%%%%%%%%%%%%%%%%%%%%%%%%%%%%%%%%%

\subsection{``Grammization'' of Schoenberg matrices}

Our goal here is to implement the ``grammization'' procedure, (see Introduction), for two
positive definite Schoenberg's matrices
\begin{equation}\label{grammiz}
\kS_X(f)=\|\exp\,(-a|x_i-x_j|^2)\|_{i,j\in\N}, \qquad \kS_X(f)=\|\exp\,(-a|x_i-x_j|)\|_{i,j\in\N}\,, \qquad a>0,
\end{equation}
and also for a certain family of Schoenberg's matrices which contains the second one in \eqref{grammiz}.

The key observation is stated as the following lemma.
\begin{lemma}\label{lem4.11}
Let $f\in L^2(\R^n)$ and $f_\xi:=f(\cdot-\xi)$ be its shift on
$\xi\in\R^n$. Then for any $\xi,\eta\in\R^n$
  \begin{equation}\label{genergram}
\langle f_{\xi},f_{\eta}\rangle_{L^2(\R^n)} = \widehat{F}(\xi-\eta),
\qquad F(t):=(2\pi)^{n/2}|\widehat{f}(t)|^2.
  \end{equation}
\end{lemma}
%%%%%%%%%%%%%%%%%%%%%%%%%%%%%%
  \begin{proof}
Since
$$
\widehat{f_\xi}(t)=\frac1{(2\pi)^{n/2}}\int_{\R^n} f_\xi(x)e^{-i(x,t)}\,dx=\widehat{f}(t)e^{-i(t,\xi)},
$$
we have by Parseval's equality
$$
 \langle f_{\xi}, f_{\eta}\rangle_{L^2(\R^n)} =\langle\widehat{f_{\xi}}, \widehat{f_{\eta}}\rangle_{L^2(\R^n)}
=\int_{\R^n} |\widehat{f}(t)|^2 e^{-i(t,\xi-\eta)}\,dt =
(2\pi)^{n/2}\widehat{F}(\xi-\eta), \quad \xi,\eta\in\R^n,
$$
as claimed.
     \end{proof}
%%%%%%%%%%%%%%%%%%%%%%%%%%%%%%%%%%%%%%%%%%%

\begin{proposition}\label{gramm1}
Let $\xi,\eta\in\R^n$, $a>0$. Then
  \begin{equation}\label{gramshoen}
e^{-\frac{a}2\,|\xi-\eta|^2}=\biggl(\frac{2a}{\pi}\biggr)^{n/2}\,\langle h_{a,\xi}, h_{a,\eta}\rangle_{L^2(\R^n)}, \qquad
h_{a,\xi}(x)=e^{-a|x-\xi|^2}.
  \end{equation}
The grammization of the first Schoenberg's matrix in \eqref{grammiz} reads as follows
\begin{equation}\label{grammizphi}
\|\exp\,\Bigl(-\frac{a}2\,|x_i-x_j|^2\Bigr)\|_{i,j\in\N}=\biggl(\frac{2a}{\pi}\biggr)^{n/2}\,Gr(\{f_j\}, L^2(\R^n)), \quad
f_j(x)=e^{-a|x-x_j|^2}.
\end{equation}
   \end{proposition}

\begin{proof} Combining Lemma  \ref{lem4.11}  (see \eqref{genergram}) with  the
well-known formula
\begin{equation*}%\label{fourier1}
\widehat{e^{-b|\cdot\,|^2}}(t)=\frac1{(2\pi)^{n/2}}\int_{\R^n}
e^{-b|x|^2-i(x,t)}\,dx =\frac1{(2b)^{n/2}}\,e^{-\frac{|t|^2}{4b}}\,,
\qquad b>0,
\end{equation*}
yields the result.
\end{proof}
%%%%%%%%%%%%%%%%%%%%%%%%%%%%%%%%%%%%%

The grammization of the second  Schoenberg matrix in \eqref{grammiz}
is similar but technically more involved.

We begin with the brief reminder of the modified Bessel functions $K_\mu$ of the second kind of order $\mu$, which
solve the differential equations
$$
t^2u''(t) + tu'(t)-(t^2 + \mu^2)u(t)=0, \qquad t>0, \quad \mu\in\R.
$$
The asymptotics for $K_\mu$ is well known (see \cite[(9.6.8)--(9.6.9)]{Abr-Ste}, \cite[p.202, (1)]{Wat})
\begin{equation}\label{asymbes}
\begin{split}
K_\mu(t)&= \left\{
  \begin{array}{ll}
    \frac{\Gamma(\mu)}2\,\left(\frac{t}2\right)^{-\mu} + O(t^{-\mu+2}), & \hbox{$\mu>0$;} \\
    \log\frac2{t} + O(1), & \hbox{$\mu=0$;}
  \end{array}
\right.
 \quad t\to 0, \\
K_\mu(t)&=\sqrt{\frac{\pi}{2t}}e^{-t}(1+O(t^{-1})), \quad
t\to\infty.
\end{split}
\end{equation}

The functions $K_\mu$ are known to satisfy $K_{-\mu}=K_\mu$ and to admit the integral representations (see, e.g., \cite[p.172, (4),(5)]{Wat})
\begin{equation}\label{modbess}
\begin{split}
K_\mu(z) &=\frac{\sqrt{\pi}}{\Gamma\bigl(\mu+\frac12\bigr)}\,\biggl(\frac{z}2\biggr)^\mu\,\int_0^\infty e^{-z\cosh r}\sinh^{2\mu}(r)dr \\ &=
\frac{\sqrt{\pi}}{\Gamma\bigl(\mu+\frac12\bigr)}\,\biggl(\frac{z}2\biggr)^\mu\,
\int_1^\infty e^{-zt}(t^2-1)^{\mu-\frac12}\,dt, \qquad \mu>-\frac12, \quad |\arg z|<\frac{\pi}2.
\end{split}
\end{equation}
Clearly, $K_\mu$ is positive and monotone decreasing function on $\R_+$.

\begin{proposition}\label{gramm2}
Let $n\ge2$ and $K_\mu$ be the modified Bessel function of the second kind of order $\mu$, $0\le\mu<n/4$. For $a>0$ put
\begin{equation}\label{rieszradial}
f_{a,\mu}(x):=\biggl(\frac{a}{|\,x|}\,\biggr)^{\mu}\,K_\mu(a|\,x|), \quad f_{a,\mu,\xi}(x):=f_{a,\mu}(x-\xi), \qquad x,\xi\in\R^n.
\end{equation}
Then with $p:=\frac{n}2-2\mu>0$ the following equality holds for all $\xi,\eta\in\R^n$
\begin{equation}\label{grambern}
\biggl(\frac{|\xi-\eta|}{a}\biggr)^p\,K_p\,(a|\xi-\eta|)=
\frac{2^{\frac{n}2-2\mu}}{\pi^{\frac{n}2}\,B\bigl(\frac{n}2-\mu, \frac12\bigr)}\,
\langle f_{a,\mu,\xi}, \,f_{a,\mu,\eta}\rangle_{L^2(\R^n)}\,.
\end{equation}
\end{proposition}
\begin{proof}
It follows from \eqref{asymbes} that $f_{a,\mu}\in L^1(\R^n)\cap L^2(\R^n)$ for $0\le\mu<n/4$.
We begin with the formula for the Fourier transform
\begin{equation}\label{furtrans}
\widehat{f_{a,\mu}}(t)=\frac{2^{q-\mu}\,\Gamma(q-\mu+1)}{(a^2+|t|^2)^{q-\mu+1}}\,.
\end{equation}
It is likely to be known, but due to its importance for the sequel, we outline the proof.

As $f_{a,\mu}$ is a radial function, then so is its Fourier transform $\widehat{f_{a,\mu}}(\cdot)=F_{a,\mu}(|\cdot|)$
and by \eqref{radialft},
$$ F_{a,\mu}(r)=\frac1{r^q}\int_0^\infty J_q(rs)s^{q+1}f_{a,\mu}(s)ds=
\frac{a^\mu}{r^q}\int_0^\infty J_q(rs)K_\mu(as)s^{q-\mu+1}\,ds, \quad q=\frac{n}2-1. $$
The latter integral is known in the theory of Bessel functions as (see \cite[p.410, (1)]{Wat})
\begin{equation*}
\begin{split}
\int_0^\infty & J_q(rs)K_\mu(as)s^{-\lambda}ds=
\frac{\Gamma(\frac{q-\lambda+\mu+1}2)\Gamma(\frac{q-\lambda-\mu+1}2)}{2^{\lambda+1}\Gamma(q+1)}\,\frac{r^q}{a^{q-\lambda+1}}\times \\
& F\biggl(\frac{q-\lambda+\mu+1}2\,, \,\frac{q-\lambda-\mu+1}2\,; \,q+1; \,-\frac{r^2}{a^2}\biggr), \quad q-\lambda+1>\mu,
\end{split}
\end{equation*}
$F$ is the Gauss hypergeometric function. The calculation with $\lambda=\mu-q-1$ gives
$$ q-\lambda+\mu+1=2(q+1), \quad q-\lambda-\mu+1=2(q-\mu+1)=n-2\mu>0, $$
so
$$ F_{a,\mu}(r)=2^{q-\mu}\,\Gamma(q-\mu+1)\, a^{2(\mu-q-1)}\,F\biggl(q+1, q-\mu+1; q+1; \,-\frac{r^2}{a^2}\biggr). $$
The known formula for the hypergeometric series
$$ F\biggl(q+1,q-\mu+1;q+1;\,-\frac{r^2}{a^2}\biggr)=\frac{a^{2(q-\mu+1)}}{(a^2+r^2)^{q-\mu+1}} $$
leads to \eqref{furtrans}.

To apply \eqref{genergram} it remains to compute
$$ \widehat{F}(t)=\frac1{(2\pi)^{n/2}}\,\int_{\R^n}|\widehat{f_{a,\mu}}(u)|^2 e^{-i(t,u)}\,du=
\frac1{(2\pi)^{n/2}}\,\int_{\R^n}\frac{e^{-i(t,u)}}{(a^2+|\,u|^2)^{2(q-\mu+1)}}\,du.
$$
The latter Fourier transform is known (see, e.g., \cite[Theorem 6.13]{Wend05}) and can be computed, for instance,
by using again \eqref{radialft} and \cite[p.434, (2)]{Wat}
\begin{equation}\label{fourtran1}
g(x)=G(|x|), \quad G(r)=\frac1{r^q}\int_0^\infty \frac{s^{q+1}J_q(rs)\,ds}{(a^2+s^2)^{2(q-\mu+1)}}=
\biggl(\frac{r}{a}\biggr)^{q-2\mu+1}\frac{K_{2\mu-q-1}(ar)}{2^{2q-2\mu+1}\,\Gamma(2(q-\mu+1))}\,. \end{equation}
Since
$$ q-2\mu+1=\frac{n}2-2\mu=p>0, \qquad K_{-p}=K_p, $$
we have
\begin{equation*}
\langle f_{a,\mu,\xi}, \,f_{a,\mu,\eta}\rangle_{L^2(\R^n)}=\frac{(2\pi)^{n/2}\Gamma^2\bigl(\frac{n}2-\mu\bigr)}{2\Gamma(n-2\mu)}\,
\biggl(\frac{|\xi-\eta|}{a}\biggr)^p\,K_p\,(a|\xi-\eta|\,).
\end{equation*}
But $\Gamma(2z)=2^{2z-1}\pi^{-1/2}\,\Gamma(z)\Gamma(z+1/2)$ and \eqref{grambern} follows.
\end{proof}
%%%%%%%%%%%%%%%%%%%%%%%%%%%%%%%%%%%%%%%%%%%%%%%%%
   \begin{corollary}\label{cor3.13}
The grammization for the second Schoenberg matrix in \eqref{grammiz}
is
\begin{equation}\label{grammizexp}
\begin{split}
\|\exp\,(-a|x_j-x_k|)\|_{j,k\in\N} &=Gr(\{g_j\}, L^2(\R^n)), \\
\quad g_j(x) &=\sqrt{\frac{2\Gamma\bigl(\frac{n+3}4\bigr)a}{\pi^{\frac{n+2}2}\,\Gamma\bigl(\frac{n+1}4\bigr)}}\,
\biggl(\frac{a}{|x-x_j|}\,\biggr)^{\frac{n-1}4}\,K_{\frac{n-1}4}(a|x-x_j|).
\end{split}
\end{equation}
In particular,
\begin{equation}\label{malshmu}
 e^{-a|\xi-\eta|}=\frac{a}{2\pi}\,\int_{\R^3} \frac{e^{-a|x-\xi|}}{|x-\xi|}\,\frac{e^{-a|x-\eta|}}{|x-\eta|}\,dx,
\qquad \xi,\eta\in\R^3, \quad a>0.
\end{equation}
\end{corollary}
%%%%%%%%%%%%%%%%%%%%%%%%%%%%%%%%%%%%%%%%%%%
\begin{proof}
Take $\mu=\frac{n-1}4$, so $p=1/2$, and the function in the left side of \eqref{grambern} is just
the exponential function \cite[p.80, (13)]{Wat}
  \begin{equation}\label{4.27}
\sqrt{\frac{|\xi-\eta|}{a}}\,K_{1/2}(a|\xi-\eta|)=\sqrt{\frac{\pi}{2}}\,\frac{e^{-a|\xi-\eta|}}{a},
    \end{equation}
which is \eqref{grammizexp}.

If $n=3$, $\mu=1/2$, then
\begin{equation} \label{malshmu1}
f_{a,1/2,\xi}(x)=\biggl(\frac{a}{|x-\xi|}\,\biggr)^{1/2}\,K_{1/2}(a|x-\xi|)=\sqrt{\frac{\pi}2}\,\frac{e^{-a|x-\xi|}}{|x-\xi|},
\end{equation}
and \eqref{malshmu} follows.
\end{proof}

Note that \eqref{malshmu} is one of the cornerstones of \cite{MalSch12} (see formula (3.26) in there).

The case $n=2$, $\mu=0$ leads to the following

\begin{corollary}\label{cor3.14}
For all $\xi,\eta\in\R^2$ and $a>0$
\begin{equation*}
\frac{|\xi-\eta|}{a}\,K_1\,(a|\xi-\eta|)=
\frac1{\pi}\,\langle K_0(a|\cdot-\xi|), \,K_0(a|\cdot-\eta|)\rangle_{L^2(\R^2)}.
\end{equation*}
\end{corollary}

\smallskip

There is another natural way to view \eqref{grambern}. For arbitrary $p>0$ and $a>0$ consider
the Whittle--Mat\'ern function (cf. \eqref{3.21})
 \begin{equation}\label{3.27}
 M_{p,a}(r):=\biggl(\frac{r}{a}\biggr)^p\,K_p(ar), \quad r>0.
  \end{equation}
Since $K_{-p}=K_p$, the notation makes sense for negative indices,
and another family of the Whittle--Mat\'ern functions comes in
$$ \widetilde M_{p,a}(r)=M_{-p,a}(r)= \biggl(\frac{a}{r}\biggr)^p\,K_p(ar)  , \qquad p>0, \qquad \widetilde M_{0,a}(r)=K_0(r). $$
Then equality \eqref{grambern} with $0<2p\le n$ reads
\begin{equation}
\begin{split}
M_{p,a}(|\xi-\eta|) &=\langle c_{n,p}\widetilde M_{d,a}(|\cdot\,-\xi|), \,c_{n,p}\widetilde M_{d,a}(|\cdot\,-\eta|)\rangle_{L^2(\R^n)}, \\
0\le d& :=\frac12\biggl(\frac{n}2-p\biggr)<\frac{n}4\,,\quad c_{n,p}^2 =\frac{2^p}{\pi^{\frac{n}2}\,B\bigl(d, \frac12\bigr)}
\end{split}
\end{equation}
for all $\xi,\eta\in\R^n$.

To have a proper normalization at the origin we put (see \eqref{asymbes} and \eqref{3.21})
\begin{equation*}
M_p(r)=\frac{M_{p,1}(r)}{2^{p-1}\Gamma(p)}=\frac{r^pK_p(r)}{2^{p-1}\Gamma(p)}=1+O(r^2), \quad r\to 0.
\end{equation*}
As a byproduct of Proposition \ref{gramm2} we have (cf. \cite{gutgne}, \cite[Table 2]{gn2013}).
\begin{corollary}\label{besshoen}
$M_p\in\Phi_\infty$ for all $p>0$.
\end{corollary}
\begin{proof}
Take $n>2p$. By Proposition \ref{gramm2}, for each finite set $X\subset\R^n$ the Schoenberg matrix $\kS_X(M_p)$ is the Gramm matrix,
so $\kS_X(M_p)\ge0$. Hence $M_p\in\Phi_n$ for all such $n$, as claimed.
\end{proof}

With regard to Corollary \ref{besshoen} one might ask whether the functions $M_p$ belong to certain subclasses of
$\Phi_\infty$, for instance, to the class $CM_0(\R_+)$ of completely monotone functions.
The result below seems interesting on its own.

\begin{proposition}
For the Whittle--Mat\'ern function $M_p$ the following statements hold.
\begin{itemize}
  \item [\em (i)] $M_p\in CM(\R_+)$ if and only if $-\infty<p\le 1/2$.
  \item [\em (ii)] $M_p\in CM_0(\R_+)$ if and only if $0<p\le 1/2$.
\end{itemize}

\end{proposition}
\begin{proof} The assertion for $-\infty<p<1/2$ follows directly from the second integral representation \eqref{modbess}
and the Bernstein theorem, if one puts $\nu=-p$. Note that the Bernstein measure is finite if and only if $0<p<1/2$.
For $p=1/2$ we have
$$ M_{1/2}(r)=e^{-r}\in CM_0(\R_+). $$

Let now $p>1/2$. We wish to show that inequalities \eqref{derivativescm} are violated for some $k\ge1$. The argument relies on
the differentiation formulae for the Bessel functions, which in our notation look as (see \cite[p.74]{Wat})
\begin{equation}\label{differentiation}
\left(\frac1{z}\,\frac{d}{dz}\right)^m\, M_{p,1}(z)=(-1)^mM_{p-m,1}(z).
\end{equation}
For $m=1$ it displays the fact that $M_{p,1}$ is monotone decreasing function on $\R_+$. For $m=2$ we have
$$ M_{p,1}''(r)=-M_{p-1,1}(r)+r^2M_{p-2,1}(r). $$

For $p\ge2$ obviously $r^2M_{p-2,1}\to 0$ as $r\to +0$, so $M_{p,1}''(+0)=-2^{p-2}\Gamma(p-1)<0$, which is inconsistent with
\eqref{derivativescm} for $k=2$. If $1<p<2$, then again
$$ r^2M_{p-2,1}(r)=r^pK_{2-p,1}(r)=r^{2p-2}M_{2-p,1}(r)\to 0, \qquad r\to +0, $$
with the same conclusion.

Finally, let $1/2<p<1$. From \eqref{differentiation} with $m=1$ one has
$$ M_{p,1}'(r)=-rM_{p-1,1}(r)=-r^pK_{1-p}(r)=-r^{2p-1}M_{1-p,1}(r)\to 0, \qquad r\to +0 $$
so $M_{p,1}'(+0)=0$ that is impossible for a nonconstant completely monotone function. The proof is complete.
\end{proof}

\begin{remark}
For $0\le p\le1/2$ a stronger result is proved in \cite{milsam}, namely, $e^rM_{p,1}(r)\in CM(\R_+)$. Our results for the other values
of $p$ seem to be new.
\end{remark}

\subsection{Minimality conditions and Riesz sequences in $L^2(\R^n)$}\label{rieszsequences}

The classical result of Bari (see, e.g., \cite[Theorem 6.2.1]{GK69}, \cite[p.170]{Nik2}) states that a sequence $\{\varphi_k\}_{k\in\N}$
of vectors in a Hilbert space is a Riesz sequence if and only if the corresponding Gramm matrix $Gr\{\varphi_k\}_{k\in\N}$
generates a bounded and invertible linear operator on $\ell^2$.
We examine here certain systems of shifted functions from this viewpoint.

The definitions below are standard (cf. \cite[Chapter VI]{GK69}).
\begin{definition}\label{def_Minimality}
A sequence of vectors $\{f_j\}_{j\in\N}$ in a Hilbert space $\kH$ is
called {\it minimal}, if neither of $f_k$ belongs to the closed
linear span $\kL(\{f_j\}_{j\not=k})$ of the others. In other words,
  \begin{equation*}
\delta_k:=\dist(f_k/\|f_k\|, \kL(\{f_j\}_{j\not=k}))>0, \qquad k\in\N.
  \end{equation*}
$\{f_j\}_{j\in\N}$ is {\it uniformly minimal}, if
$\inf_k\delta_k>0$.
\end{definition}

Recall that Riesz--Fischer systems are defined in \eqref{rieszfischerse}.
  \begin{lemma}\label{lem-Unif_Minimality}
Any Riesz--Fischer sequence $\{f_j\}_{j\in \N}$ is uniformly minimal.
   \end{lemma}
\begin{proof}
It is clear that a Riesz--Fischer sequence is bounded from below, that is, $\|f_j\| \ge c$, $j\in\N$.
By Definition \ref{def4.1}(i), (see \eqref{rieszfischerse}), for any fix $j$ and any finite sequence
$\{\xi_{k}\}\subset \C$
   \begin{equation}
\bigg\|\sum_{k\not = j}\xi_{k}f_k - f_j\bigg\|^2   \geq  c~ \left(c
+ \sum_{k\not =j} |\xi_k|^2\right) \ge c^2,
  \end{equation}
so by Definition \ref{def_Minimality} $\{f_j\}_{j\in \N}$ is uniformly minimal, as claimed.
\end{proof}

Given a function $f\in L^2(\R^n)$ and a set
$X=\{x_j\}_{j\in\N}\subset\R^n$, consider a sequence of the shifted
functions $\kF_X(f)=\{f(\cdot-x_j)\}_{j\in\N}$. Denote
$f_j(\cdot)=f(\cdot-x_j)$.
%%%%%%%%%%%%%%%%%%%%%%%%%%%%%%%%%%%%%%%%%%%%%%%%%%%%
%
%
   \begin{proposition}[=Proposition \ref{Intro prop4.21}]\label{prop4.21}
Let $f\in L^2(\R^n)$ be a real-valued and radial function such that
$\widehat f\not=0$ a.e. Then the following statements are equivalent.
\begin{itemize}
  \item [\em (i)]   $\kF_X(f)$  forms a Riesz--Fischer  sequence in $L^2(\R^n);$
  \item [\em (ii)] $\kF_X(f)$ is uniformly minimal in $L^2(\R^n);$
  \item [\em (iii)]  $X$ is a separated set, i.e., $d_*(X)>0$.
\end{itemize}
     \end{proposition}
%%%%%%%%%%%%%%%%%%%%%%%%%%%%%%%%%%%%%%%%%%%%%%%%%
   \begin{proof}
Implication (i)$\Rightarrow$(ii) is immediate from Lemma  \ref{lem-Unif_Minimality}.

(ii)$\Rightarrow$(iii).  With no loss of generality we can assume
that $\|f\|_{L^2(\R^n)}=1$, so $\|f_j\|_{L^2(\R^n)}=1$ for all
$j\in\N$. The normalization in \eqref{genergram} shows that
$\widehat F(0)= \|f\|^2_{L^2(\R^n)} =1$.

Let $\kF_X(f)$ be uniformly minimal. Then there exists $\varepsilon>0$ such that $\|f_j-f_k\|^2\ge 2\varepsilon$ for all
$j\not = k\in\Bbb N$. A combination of the latter inequality with identity \eqref{genergram} yields
  \begin{equation} 1- \widehat F(|x_j-x_k|) =
1- \langle f_j,f_k\rangle_{L^2(\R^n)} = \frac{\|f_j-f_k\|^2}2 \ge
\varepsilon, \quad j,k \in \N,
    \end{equation}
and so $d_*(X)>0$ follows.

(iii)$\Rightarrow$(i). Let $d_*(X)>0$. As all functions in question are radial, we put
\begin{equation}\label{radial1}
F(t)=(2\pi)^{n/2}|\widehat f(t)|^2 = F_0(|t|), \qquad \widehat F(t)=\widetilde F_0(|t|).
\end{equation}
Clearly,  $F\ge 0$ a.e. on $\R^n$ and $F\in L^1(\R^n)$ since $f\in
L^2(\R^n)$. Hence, by the inversion formula,
$$
\widehat{F}(\xi) =
(2\pi)^{-n/2}\int\nolimits_{\R^n}e^{-i(t,\xi)}F(t)dt=(2\pi)^{-n/2}\int\nolimits_{\R^n}e^{i(t,\xi)}F(t)dt,
$$
so $\widehat{F}$ is a radial positive definite function, i.e.,
$\widetilde F_0\in\Phi_n$. We see that the measure $\mu =
\mu_{\widehat{F}}$ from the Bochner representation \eqref{bochequa}
of $\widehat{F}$ is absolutely continuous, $\mu_{\widehat{F}}=
(2\pi)^{-n/2}F\,dt$. Moreover, the condition $\widehat f\not=0$ a.e.
implies $F>0$ a.e. on $\R^n$, that is, $\mu_{\widehat{F}}$ is
equivalent to the Lebesgue measure $dt$ on $\R^n$. Hence, the
representing Schoenberg measure $\nu = \nu_{\widetilde F_0}$ from
\eqref{schoenberg1} is equivalent to the Lebesgue measure on $\R_+$
due to the relation $\nu\{[0,r]\}=\mu\{|x|\le r\}$ between $\nu$ and
$\mu$. Thereby the conditions of Theorem
\ref{propositionstronglydef} are met and the function $\widetilde
F_0$ is strongly $X$-positive definite.
By Lemma \ref{lem4.11} (see identity \eqref{genergram}) and
Definition \ref{def-X-posit_definit} of strongly $X$-positive
definite functions, the latter amounts to saying that
$\kF_X(f)$ is the Riesz--Fischer system. The proof is complete.
   \end{proof}
Under certain additional assumptions on $f$ we come to Riesz sequences of the shifted functions.
      \begin{theorem}[=Theorem \ref{theorem3.25}]\label{mainriesz}
Let $f\in L^2(\R^n)$ be a real-valued and radial function such that
its Fourier transform $\widehat f\not=0$ a.e.. Let $F$ and $F_0$ be defined in  \eqref{radial1}
and assume that for some majorant $h\in\kM_+$ $\eqref{function}$ the relations
\begin{equation}\label{majorant}
|\widetilde F_0(s)|\le h(s), \qquad s^{n-1} h(s)\in L^1(\R_+)
\end{equation}
hold. Then the following statements are equivalent.
\begin{itemize}
  \item [\em (i)]   $\kF_X(f)$  forms a Riesz sequence in $L^2(\R^n);$
  \item [\em (ii)]  $\kF_X(f)$ forms a basis in its linear span$;$
  \item [\em (iii)] $\kF_X(f)$ is uniformly minimal in $L^2(\R^n);$
  \item [\em (iv)]  $X$ is a separated set, i.e., $d_*(X)>0$.
\end{itemize}
\end{theorem}
\begin{proof}
The implications (i)$\Rightarrow$(ii)$\Rightarrow$(iii) are obvious. The implication (iii)$\Rightarrow$(iv) is proved in Proposition \ref{prop4.21}.

It remains to prove that (iv) implies (i). Lemma \ref{lem4.11} is a key ingredient of the proof. Condition \eqref{genergram} now reads
 \begin{equation}\label{Gram_matrix}
Gr(\{f_j, L^2(\R^n)\}) = \kS_X(\widetilde F_0).
   \end{equation}
In view of the aforementioned theorem of Bari we need to show that
under the hypothesis of Theorem \ref{mainriesz} the Schoenberg
operator $S_X(\widetilde F_0)$ is bounded and invertible.

First, assumption \eqref{majorant} implies the boundedness of
$S_X(\widetilde F_0)$ in view of Proposition \ref{nonmonoton}, and
$\kF_X(f)$ is the Bessel sequence.

Secondly, according to Proposition \ref{prop4.21},   the condition
$\widehat f\not=0$ a.e.  ensures that $\kF_X(f)$ is the Riesz--Fischer
sequence, i.e., the operator $S_X(\widetilde F_0)$ is invertible, as
claimed. Thus, by \eqref{Gram_matrix}  the Gramm matrix $Gr(\{f_j,
L^2(\R^n)\})$ is bounded  and invertible, and the Bari theorem completes the proof.
    \end{proof}
\begin{remark}
One can avoid using the Fourier transform when computing $\widetilde F_0$ from \eqref{radial1} since
\begin{equation}
\widehat F(t)=\int_{\R^n} f(t+y)f(y)dy=\widetilde F_0(|t|).
\end{equation}
\end{remark}

\begin{example}\label{examriesz}
The conditions of Theorem \ref{mainriesz} can be verified for the systems we already encountered in the previous section.
For instance, as we have seen in Proposition \ref{gramm1},
$$ f(x)=e^{-a|x|^2} \Longrightarrow \widetilde F_0(r)=\Bigl(\frac1{4a}\Bigr)^{n/2}\,e^{-\frac{a}2r^2}. $$
Similarly, it is shown in Proposition \ref{gramm2} that
$$ f(x)=\biggl(\frac{a}{|\,x|}\,\biggr)^{\mu}\,K_\mu(a|\,x|), \ \ 0\le\mu<\frac{n}4 \Longrightarrow
\widetilde F_0(r)=\frac{B\bigl(\frac{n}2-\mu,
\frac12\bigr)}{2^{n-2\mu}}\,\biggl(\frac{r}{a}\biggr)^p\,K_p(ar), \
\ p=\frac{n}2-2\mu. $$ Since in both cases $\widetilde
F_0\in\Phi_\infty\subset\kM_+$ (cf. Corollary \ref{besshoen}) and
$\widetilde F_0$ decays exponentially fast (see  \eqref{asymbes}),
Theorem \ref{mainriesz} applies, so $\kF_X(f)$ is the Riesz sequence
for each $X\in\kX_n$.
\end{example}

In view of applications in the spectral theory let us single out two particular cases of the above example.
%%%%%%%%%%%%%%%%%%%%%%%%%%%%%%%%%%%%%%
  \begin{corollary}\label{RieszSeq_exponents}
Let  $\kF_2 = \{K_0(a|\cdot- x_j|\}_{j\in \N}$ and $\kF_3=\left\{\frac{e^{-a|\cdot -x_j|}}{|\cdot - x_j|}\right\}_{j\in \N}$.
Then each of the  sequences $\kF_2$ and $\kF_3$ forms a Riesz sequence in $L^2(\R^2)$ and $L^2(\R^3)$, respectively,
for each $X\in\kX_n$.
  \end{corollary}
%%%%%%%%%%%%%%%%%%%%%%%%%%%%%%%%%%%%%%%%%%%%%%%%
%

We show now that a sequence $\kF_X(f)$ can be \emph{minimal}
but \emph{not uniformly minimal}, (so necessarily $d_*(X)=0$),
whenever  $\widehat f\not=0$ a.e. is replaced by the stronger assumption \eqref{belowfur}. Note that in the following
proposition a function $f$ is not even assumed to be radial.

\begin{proposition}\label{prop_minimality}
Given $f\in L^2(\R^n)$, assume that its Fourier transform $\widehat
f$ admits the bound
\begin{equation}\label{belowfur}
|\widehat f(t)|\ge C(1+|t|)^{-p}
 \end{equation}
for some $p>0$. Then the system $\kF_X(f)=\{f(\cdot-x_j)\}_{j\in\N}$
is minimal in $L^2(\R^n)$ for each set
$X=\{x_j\}_{j\in\N}\subset\R^n$ with no finite accumulation points.
\end{proposition}
%%%%%%%%%%%%%%%%%%%%%%%%%%%%%%%%%%%%%%%%
\begin{proof}
Denote $f_j(\cdot)=f(\cdot-x_j)$. Since the Fourier transform  is a
unitary operator in $L^2(\R^n)$, the system  $\{f_j\}_{j\in\N}$ is
minimal in $L^2(\R^n)$  if and only if so is the system of their
Fourier images  $\{\widehat f_j\}_{j\in\N}$. Note that $\widehat
f_j=\widehat f\,e^{-i(\cdot,x_j)}$, $f=f_1$ (recall that $x_1=0$).
To prove the minimality of $\{\widehat f_j\}_{j\in\N}$, it suffices
(in fact is equivalent) to construct a biorthogonal sequence
$\{h_j\}_{j\in\N}$,
\begin{equation*}
\langle h_j, \widehat f_k \rangle_{L^2(\R^n)}=\int_{\R^n}
h_j(t)\overline{\widehat f(t)}\,e^{i(t,x_k)}dt=\delta_{kj}, \qquad
h_j\in\kL(\{\widehat f_j\}_{j\in\N}).
\end{equation*}
To this end take a smoothing function $u$ and its shifts $u_j$
\begin{equation*}
u(x):=\left\{
        \begin{array}{ll}
         \exp\Bigl(\frac{|x|^2}{|x|^2-1}\Bigr), & |x|\le1; \\
          0, & |x|>1.
        \end{array}
      \right. \qquad u_j(x):=u\biggl(\frac{x-x_j}{\rho_j}\biggr), \quad \rho_j:=\dist(x_j, X\backslash\{x_j\})>0
\end{equation*}
for each $j$, since $X$ has no finite accumulation points. By the
definition $u_j(x_k)=\delta_{kj}$. Since $u\in C_0^\infty$
(infinitely differentiable with compact support), then both $u_j$
and $\widehat u_j$ belong to the Schwartz class. Define
\begin{equation*}
h_{j,1}(t):= (2\pi)^{-\frac{n}2}\,\frac{\widehat
u_j(t)}{\overline{\widehat f(t)}}=
(2\pi)^{-\frac{n}2}\,\frac{\rho_j^n\widehat
u(\rho_jt)}{\overline{\widehat f(t)}}.
\end{equation*}
In view of \eqref{belowfur}, $h_{j,1}\in L^1(\R^n)\cap L^2(\R^n)$,
so
$$ \langle h_{j,1}, \widehat f_k \rangle_{L^2(\R^n)}=
\int_{\R^n} h_{j,1}(t)\overline{\widehat
f(t)}\,e^{i(t,x_k)}dt=(2\pi)^{-\frac{n}2}\,\int_{\R^n} \widehat
u_j(t)\,e^{i(t,x_k)}dt=u_j(x_k)=\delta_{kj}. $$ We are left with
putting $h_j:=\mathbb{P}h_{j,1}$, where $\mathbb{P}$ is a projection
from $L^2(\R^n)$ onto $\kL(\{\widehat f_j\}_{j\in\N})$. The proof is
complete.
\end{proof}
It is easy to construct a set $X$ with $d_*(X)=0$, which has no finite accumulation points.
%%%%%%%%%%%%%%%%%%%%%%%%%%%%%%%%%%%%%%%%%%%%%%%%%%%%
\begin{example}
Let $f=f_{a,\mu}$ \eqref{rieszradial} with $0\le\mu<n/4$. Condition
\eqref{belowfur}  follows from \eqref{furtrans}, so the system
$\kF_X(f)$ is minimal for each set $X$ of distinct points which has
no finite accumulation  points.
\end{example}
  \begin{remark}
Corollary \ref{RieszSeq_exponents} is crucial in  the study of
certain spectral properties of the Schr\"odinger operator with point
interactions \cite{MalSch12}. The statement on the system $\kF_3$
was proved  in \cite[Theorem 3.8]{MalSch12} in an absolutely
different manner. The appearance of such functions takes its origin
in the following classical  formulae for the resolvent of the
Laplace operator $H_0:=-\Delta$ in $\Bbb R^3$ and $\Bbb R^2$,
respectively,
      \begin{equation}\label{3.2SGolMal}
(H_0 - z I)^{-1}f = \frac{1}{4\pi}\int_{{\mathbb
R}^3}\frac{e^{i\sqrt{z}|x-t|}}{|x-t|}~f(t)~dt, \qquad (H_0 - z
I)^{-1}f =  \frac1{2\pi}\int_{{\mathbb R}^2}
 K_0(\sqrt{-z}|x - t|)~f(t)~dt,
      \end{equation}
(see \cite[formulae (1.1.19), (1.5.15)]{AGHH88}).  Note also that a
special case of Proposition \ref{prop_minimality} regarding
minimality of  the system $\kF_3$  was  proved in \cite[Lemma
3.5]{MalSch12} in a different manner. In the latter case
$$ f(x)=\frac{e^{-a|x|}}{|x|}\,, \qquad \widehat f(t)=\sqrt{\frac2{\pi}}\,\frac1{a^2+|t|^2}\,, $$
and \eqref{belowfur} automatically holds.
   \end{remark}
%

%%%%%%%%%%%%%%%

\quad

Leonid Golinskii, \\
\emph{Mathematics Division, Low Temperature Physics Institute, NAS of Ukraine},\\
\emph{47 Lenin ave.},\\
\emph{61103 Kharkov, Ukraine}\\
\emph{e-mail:} golinskii@ilt.kharkov.ua \\

Mark Malamud,\\
\emph{Institute of Applied Mathematics and Mechanics, NAS of Ukraine},\\
\emph{74 R. Luxemburg str.},\\
\emph{83114 Donetsk, Ukraine}\\
 \emph{e-mail:} mmm@telenet.dn.ua
 \\

Leonid Oridoroga, \\
\emph{Donetsk National University}, \\
\emph{24, Universitetskaya Str.}, \\
\emph{83055  Donetsk, Ukraine}  \\
\emph{e-mail:} oridoroga@skif.net
\\


\begin{thebibliography} {20}
%
%
\bibitem{Abr-Ste}
M. Abramovitz, I. Stegun, \textit{Handbook of Mathematical Functions}, Dover, NY 1972.

\bibitem{Akh65}
N.I. Akhiezer, \textit{The  Classical Moment Problem  and Some
Related Questions of Analysis}, Oliver  and Boyd, Edinburgh, 1965
(Russian edition: Moscow, 1961).

\bibitem{AkhGlz78}
N.I. Akhiezer, I.M. Glazman, \emph{Theory of Linear Operators in
Hilbert Spaces},  Ungar, New York, 1961.

%
\bibitem{AGHH88}
S.~Albeverio, F.~Gesztesy, R.~Hoegh-Krohn, H.~Holden, \emph{Solvable
Models in Quantum Mechanics}, Sec. Edition, (with an Appendix by P.
Exner) AMS Chelsea Publ., 2005.
%


\bibitem{BCR}
C. Berg, J.P.R. Christensen, P. Ressel, \textit{Harmonic Analysis on
Semigroups}, Springer-Verlag, New-York, 1984.

\bibitem{beram}
S. Bernstein, \textit{Sur les fonctions absolument monotones}, Acta Math.
{\bf 52} (1929), 1--66.

\bibitem{bdk}
J. Bretagnolle, C.D. Dacunha, J.L. Krivine, \textit{Lois stables et
espaces $L^p$}, Ann. Inst. H. Poincar\'e Probab. Statist., {\bf 2}
(1966), 231--259.


\bibitem{Boch}
S. Bochner, \textit{Monotone Funktionen, Stieltjessche Integrale und
harmonische  Funktionen}, Math. Ann., \textbf{108} (1933), 378-410.

\bibitem{GK69}
I.C. Gokhberg,  M.G. Krein, \textit{Introduction to the Theory of Linear
Nonselfadjoint Operators}, AMS, Providence, RI, 1969.

\bibitem{Garn}
J.B. Garnett, \textit{Bounded Analytic Functions}, Academic Press, 1981.

\bibitem{gn1998}
T. Gneiting, \textit{On $\alpha$-symmetric multivariate characteristic functions},
Journal of Multivariate Analysis, {\bf 64} (1998), 131--147.

\bibitem{gn2013}
T. Gneiting, \textit{Strictly and non-strictly positive definite functions on spheres},
Bernoulli, {\bf 19} (2013), 1327--1349.

\bibitem{gmo2}
L. Golinskii, M. Malamud, and L. Oridoroga,
\textit{On radial positive definite functions}, in preparation.

\bibitem{GMZ11}
N. Goloschapova, M. Malamud, and V. Zastavnyi,
\textit{Radial positive definite functions and spectral theory of
Schr\"odinger operators with point interactions}, Math. Nachr.
{\bf 285} (2012), no. 14-15 (2012), 1839--1859.

\bibitem{gol81}
B.I. Golubov, \textit{On Abel--Poisson type and Riesz means}, Anal. Math. {\bf 7} (1981), 161--184.

\bibitem{gutgne}
P. Guttorp, T. Gneiting, \textit{Studies in the history of probability and statistics XLIX:
On the Mat\'ern correlation family}, Biometrika {\bf 93}, (2006) 989--995.

\bibitem{MalSch12}
M.M. Malamud, K. Schm\"udgen, \textit{Spectral theory of Schr\"odinger
operators with infinitely many point interactions and radial
positive definite functions}, J.  of Functional Analysis, {\bf
v. 263}, No 10 (2012),  p. 3144-3194.


\bibitem{milsam}
K.S. Miller, S.G. Samko, \textit{Completely monotonic functions},
Integral Transforms and Special Functions, {\bf 12} (2001) 389--402.

\bibitem{Nik1}
N.K. Nikolski, \textit{Operators, Functions, and Systems: An Easy Reading},
v.1: Hardy, Hankel, and Toeplitz, AMS, Providence RI, 2002.

\bibitem{Nik2}
N.K. Nikolski, \textit{Operators, Functions, and Systems: An Easy Reading},
v.2: Model Operators and Systems, AMS, Providence RI, 2002.

\bibitem{Pel89}
V.V. Peller, \textit{When is a function of a Toeplitz operator close to
a Toeplitz operator?}, Operator Theory, Birkh\"{a}user, {\bf 42}
(1989), 59-85.

\bibitem{RoRo}
M. Rosenblum and J. Rovnyak, \textit{Hardy Classes and Operator Theory}, Oxford
Univer. Press, 1985.

\bibitem{Ros60}
M. Rosenblum,  \textit{The absolute continuity of Toeplitz's
matrices}, Pacific J. of Math. {\bf 10} (1960), 987--996.

\bibitem{Sch38}
I.J. Schoenberg, \textit{Metric spaces and completely monotone functions},
 Ann. Math. {\bf 39} (1938), 811--841.

\bibitem{Sch38_1}
I.J. Schoenberg, \textit{Metric  spaces   and  positive definite
functions}, Trans. Amer. Math. Soc. {\bf 44} (1938), 522--536.

\bibitem{Stein-Weiss}
E. Stein, G. Weiss, \textit{Introduction to Fourier Analysis on Euclidian Spaces},
Princeton University Press, NJ, 1971.

\bibitem{sun93}
X. Sun, \textit{Conditionally negative definite functions and their application to multivariate
interpolation}, J. of Approx. Theory, {\bf 74}, (1993) 159--180.

\bibitem{TriBel}
R.M. Trigub, E. Bellinsky \textit{Fourier Analysis and Approximation of Functions}
Springer Science Business Media, Kluwer, 2004.

\bibitem{Wat}
G.N. Watson, \textit{A Treatise on the Theory of Bessel Functions}, 2nd ed., Cabridge University Press,
1966.

\bibitem{WelWil}
J. Wells, L. Williams, \textit{Embeddings and Extensions in Analysis}, Springer-Verlag Berlin, Heidelberg,
New York, 1975.

\bibitem{Wend05}
H. Wendland, \textit{Scattered Data Approximation}, Cambridge University Press, 2005.

\bibitem{WhWa}
E.T. Whittaker and G.N. Watson, \textit{A Course of Modern Analysis}, part II: Principal transcendental functions,
Cambridge University Press, 1927.

\bibitem{Wid}
D. Widder, \textit{The Laplace Transform}, Princeton, 1946.

\bibitem{young80}
R.M. Young,   \textit{An Introduction to Nonharmonic Fourier Series}, Academic Press, New York, 1980.

\bibitem{zas2000}
V.P. Zastavnyi, \textit{On positive definiteness of some functions},
 J. Multiv. Anal. {\bf 73} (2000),  55--81.

\end{thebibliography}
\end{document}